%% file: spanning.tex
\documentclass[11pt]{article}
\usepackage{amsmath,amssymb,amsthm, latexsym,color,epsfig,a4, graphicx}
\usepackage{enumerate}
\usepackage{changepage}
\usepackage{tikz}
\usepackage{xspace}
\usepackage{bbding}
\usepackage{appendix}
\usepackage{subfig}

\usepackage[ruled, vlined, linesnumbered, nofillcomment]{algorithm2e}

\newtheorem{innercustomthm}{Lemma}
\newenvironment{customlemma}[1]
  {\renewcommand\theinnercustomthm{#1}\innercustomthm}
  {\endinnercustomthm}

\theoremstyle{plain}
\newtheorem{theorem}{Theorem}[section]
\newtheorem{lemma}[theorem]{Lemma}
\newtheorem{claim}[theorem]{Claim}

\newtheorem*{claim*}{Claim}

\theoremstyle{definition}
\newtheorem{definition}[theorem]{Definition}

\newcommand{\eps}{\ensuremath{\varepsilon}}

\newcommand{\Gnp}{\ensuremath{G_{n,p}}}

\newcommand{\calB}{\ensuremath{\mathcal B}}

\newcommand{\calD}{\ensuremath{\mathcal D}}
\newcommand{\calE}{\ensuremath{\mathcal E}}

\newcommand{\calH}{\ensuremath{\mathcal H}}
\newcommand{\calI}{\ensuremath{\mathcal I}}

\newcommand{\calL}{\ensuremath{\mathcal L}}

\newcommand{\calP}{\ensuremath{\mathcal P}}

\newcommand{\calT}{\ensuremath{\mathcal T}}

\date{}
\title{\vspace{-0.7cm}Spanning universality in random graphs}
\author{
	Asaf Ferber	
		\thanks{
			Department of Applied Mathematics, MIT, USA. Email: {\tt ferbera@mit.edu}. Research is partially supported by an NSF grant 6935855.
		}	
	\and
	Rajko Nenadov
		\thanks{
			School of Mathematical Sciences, Monash University, Australia. Email: {\tt rajko.nenadov@monash.edu}.
		}
}

\begin{document}
\maketitle
\begin{abstract}
A graph is said to be $\calH(n, \Delta)$-universal if it contains every graph on $n$ vertices with maximum degree at most $\Delta$. Using a `matching-based' embedding technique introduced by Alon and F\"uredi, Dellamonica, Kohayakawa, R\"odl and Ruci\'nski showed that the random graph $\Gnp$ is asymptotically almost surely $\calH(n, \Delta)$-universal for $p = \tilde \Omega(n^{-1/\Delta})$ --- a threshold for the property that every subset of $\Delta$ vertices has a common neighbour. This bound has become a benchmark in the field and many subsequent results on embedding spanning structures of maximum degree $\Delta$ in random graphs are proven only up to this threshold. We take a step towards overcoming limitations of former techniques by showing that $\Gnp$ is almost surely $\calH(n, \Delta)$-universal for $p = \tilde \Omega(n^{- 1/(\Delta-1/2)})$.
\end{abstract}

\section{Introduction}
Ever since its introduction by Erd\H{o}s and R{\'e}nyi~\cite{ErdosRenyi} in 1960, random graphs have been one of the main objects of
study in probabilistic combinatorics. Given a positive integer $n$
and a real number $p \in [0,1]$, the binomial random graph $\Gnp$ is the random variable taking values in the set
of all labelled graphs on the vertex set $[n]$. We can describe the
probability distribution of $\Gnp$ by saying that each two
elements of $[n]$ form an edge in $\Gnp$ with probability $p$, independently of all other pairs.

The core meta-problem in the area is the study of
the {\em evolution of $G_{n,p}$}, that is analysing how it behaves with respect to certain graph properties as $p$ traverses
the interval $[0,1]$. A result of Bollob\'as and Thomason~\cite{bollobas1987threshold} states that for every non-trivial monotone graph property $\mathcal P$ the random graph undergoes a sudden change from almost surely not having to almost surely having the property $\mathcal{P}$. We are interested in determining when this change happens. In short, we are interested in determining a \emph{threshold function} for $\mathcal P$. Recall that a function $q(n)$ is a threshold function for $\mathcal P$ if
$$
	 \lim_{n \rightarrow \infty} \Pr\left[\Gnp \text{ satisfies } \mathcal P\right] =
	 \begin{cases}
 	 	0&  \text{if } p(n)/q(n) \rightarrow 0,   \\
    	1&  \text{if } p(n)/q(n) \rightarrow \infty.
	 \end{cases}
$$
One of the fundamental properties is related to \emph{subgraph containment}: given a graph $H$, determine the values of $p$ for which a typical\footnote{We say that a typical $\Gnp$ satisfies a property $\mathcal{P}$ if $\lim_{n \rightarrow \infty} \Pr[\Gnp \text{ satisfies } \mathcal P] = 1$.} $\Gnp$ contains a copy of $H$.

Let us first consider the case where $H$ is a graph on a fixed number of vertices, that is of a size which does not depend on $n$. A classical result in the random graph theory states that a threshold for containing a given (fixed) graph $H$ as a subgraph is
$n^{-1/m(H)}$, where $m(H) = \max\{|E(H')|/|V(H')| \colon H' \subseteq H\}$. The case where $H$ is \emph{balanced} (that is, $m(H)=|E(H)|/|V(H)|$) has already been proven in \cite{ErdosRenyi} and the general case was solved by Bollob\'as (see the relevant chapter in \cite{bollobas1998random}).

Unfortunately, for larger graphs (that is, graphs of size that depends on $n$) such a characterisation is not known and most of the current research is focused on understanding specific families of graphs (or specific graphs). 
Perhaps the simplest and most natural candidate graph $H$ to start with is a \emph{perfect matching} (that is, a collection of $\lfloor n/2\rfloor$ pairwise disjoint edges). In their original paper, Erd\H{o}s and Renyi \cite{ErdosRenyi} showed that $n^{-1} \log n$ is a threshold for such graph $H$. More than a decade later, by introducing the so-called `rotation-extension' technique, Pos\'a \cite{posa} showed that $n^{-1} \log n$ is also a threshold for the existence of a \emph{Hamilton cycle}, that is, a cycle which passes through every vertex exactly once. Nowadays, much more precise results are known about Hamilton cycles in random graphs and for more details we refer the reader to \cite{bollobas1998random,frieze2015introduction} and references therein. Even though a great deal of effort has been made, there are only a handful of other examples of large graphs for which a threshold is known. A detailed survey of recent progress can be found in \cite{bottcher2017large}. We now briefly mention a few results which have had a significant influence in this area of research and are relevant for our main result.

One of the earliest general results on the existence of large graphs in random graphs is by Alon and F\"uredi \cite{AF}. They showed that a typical $\Gnp$ contains a graph $H$ with $n$ vertices (we call such a graph \emph{spanning}) and maximum degree $\Delta$ provided $p = \tilde\Omega(n^{-1/\Delta})$ (as usual, $\tilde\Omega(\cdot)$ means we hide $\log$ factors). Even though this bound is probably far from the threshold, their proof method served as a basis in much of the subsequent work in the area. We will come back to this point shortly. Arguably the simplest graph to describe -- and the most difficult to prove -- with maximum degree $\Delta$ is a \emph{$K_{\Delta + 1}$-factor}, that is a collection of $\lfloor n / |V(H)| \rfloor$ vertex disjoint copies of complete graphs with $\Delta + 1$ vertices. Following some initial progress by Krivelevich \cite{krivelevich1997triangle} and Kim \cite{kim2003perfect}, Johansson, Kahn and Vu \cite{johansson2008factors} 
showed (among other things) that
\begin{equation}
  \label{eq:JKV bound}
  \left(n^{-1}\log^{1/\Delta}n\right)^{\frac{2}{\Delta+1}}
\end{equation}
is a threshold for the existence of a $K_{\Delta + 1}$-factor. Further progress on determining a threshold for an arbitrary spanning graph $H$ was achieved by Riordan \cite{riordan2000spanning}. In particular, the main result from \cite{riordan2000spanning} shows that $p \ge n^{-2/(\Delta + 1) + \eps(\Delta)}$ suffices for any spanning graph $H$ with maximum degree $\Delta$, for some $\eps(\Delta) > 0$ which goes to $0$ as $\Delta$ goes to infinity. It is believed that in fact $\eps(\Delta) = 0$ suffices, that is the bound in \eqref{eq:JKV bound} determines an upper bound on the threshold for the appearance of any such graph $H$. This is supported by a recent result of Ferber, Luh and Nguyen \cite{ferber2016embedding} where they showed that this is indeed the case if $H$ has at most $(1 - \eps)n$ vertices (we call such graphs \emph{almost-spanning}). Of course, there are graphs for which $p$ can be much lower, such as the empty graph, however the example of a $K_{\Delta+1}$-factor shows that it is the best possible general bound for the family of graphs on at most $n$ vertices and of maximum degree at most $\Delta$. Throughout the paper, we denote this family by $\mathcal H(n,\Delta)$. 

An important subfamily of $\mathcal H(n,\Delta)$ is the family of $d$-\emph{degenerate} graphs. Recall that a graph $H$ is $d$-degenerate if there exists a labelling $V(H)=\{v_1,\ldots,v_t\}$ of its vertices such that each vertex $v_i$ has at most $d$ neighbours within $\{v_1,\ldots, v_{i-1}\}$. It follows again from the result of Riordan~\cite{riordan2000spanning} that $\Gnp$ contains a given $d$-degenerate spanning graph $H$ provided $p=\omega(n^{-1/d})$ and $d \ge 3$ (with a minor restriction on the maximum degree of $H$). As before, there are graphs with a smaller threshold, like the empty graph, but there are also examples for which such a bound on $p$ coincides with a threshold. The $d$-power of a path is one such example. In the case where $d = 2$, the only graph studied so far is the square of a Hamilton path or, more generally, the square of a Hamilton cycle (which is an `almost' $2$-degenerate graph; see \cite{bennett2016square,kuhn2012posa,nenadov17square}). Finally, we come to the most notable case  where $d = 1$. Note that a graph is $1$-degenerate if and only if it is a forest. Some partial results for such graphs (forests) were obtained in \cite{hefetz2012sharp,krivelevich2010embedding} and it was only a recent breakthrough of Montgomery \cite{montgomery2014embedding} that determined $n^{-1} \log n$ to be a threshold. We remark that there is a loss of a few $\log$s in \cite{montgomery2014embedding} but the author has recently announced an optimal bound. Moreover, Montgomery has actually announced a stronger statement -- rather than just containing one such forest, the random graph contains all of them simultaneously. In other words, a typical $\Gnp$ is \emph{universal} for such a family of graphs. The notion of universality is the main topic of this paper.

Given a family of graphs $\mathcal H$, a graph $G$ is \emph{universal} for $\mathcal H$
(or simply $\mathcal H$-universal) if it contains a copy of every graph $H \in \mathcal H$. Once we know that $\Gnp$ contains any given graph $H \in \calH$ with high probability for some $p$, the question which naturally follows is whether it contains all of them \emph{simultaneously}. In other words, we are interested in determining a threshold for the property of `being $\calH$-universal'. 

A family that has received considerable attention in recent years is $\mathcal T(n,\Delta)$, the family of all spanning forests with maximum degree at most $\Delta$. Progress towards determining a threshold for $\calT(n, \Delta)$-universality was done in \cite{ferber2015universality,johannsen2013expanders} and Montgomery \cite{montgomery2014embedding} has recently proved that a typical $\Gnp$ is $\mathcal T(n,\Delta)$-universal provided $p=\Omega(\log^c n/n)$ (and as mentioned above, he has also announced on such a result for $p=\Theta(\log n/n)$). Montgomery's proof relies on the following `simple' structure of trees, observed by Krivelevich \cite{krivelevich2010embedding}: 
each $T\in \mathcal T(n,\Delta)$ either has many \emph{leaves} (vertices of degree $1$) or many \emph{long} induced paths. We remark that the corresponding almost-spanning case was solved earlier in \cite{alon2007embedding,balogh2010large}.

As a next step, it is interesting to consider families of more `complicated' graphs such as $d$-degenerate graphs, which are a generalisation of forests. Let us denote by $\calH(n, d, \Delta)$ the family of all $d$-degenerate graphs on at most $n$ vertices and  maximum degree at most $\Delta$ and note that $\calT(n, \Delta) = \calH(n, 1, \Delta)$.  Already for $d \ge 2$ we are not aware of any simple structure which characterises $d$-degenerate graphs, such as the one described in the case of forests/trees. Consequently, the corresponding universality problem is still far from being settled. As mentioned earlier, for $d \ge 3$ Riordan's result shows that $\Gnp$ contains one such graph provided $p = \omega(n^{-1/d})$. However, as his proof is based on a second-moment argument it does not translate into a universality statement (the bounds on the probability of not containing one such graph are too large for a union bound).
The current best bound is $p=\tilde\Omega(n^{-1/(2d+1)})$ by Allen et al. \cite{allen2016blow}, obtained as a corollary of a general \emph{sparse blow-up lemma} that they have developed. As a warm up for our main result, we slightly improve the bound of Allen et al. by proving the following:

\begin{theorem} \label{thm:spanning_deg}
	Let $d\leq \Delta$ be positive integers and let
	$$
		p \ge \left( n^{-1}\log^3 n \right)^{\frac{1}{2d}}.
	$$
	Then $G_{n,p}$ is w.h.p\footnote{With high probability, i.e. with probability tending to $1$ as $n \to \infty$.}  $\calH(n, d, \Delta)$-universal.
\end{theorem}

A heuristic argument for why our bound is a natural one to start with can be explained as follows: Consider a $d$-degenerate graph $H$ on $n$ vertices. It follows from the definition of being $d$-degenerate that $e(H) \leq dn$ (and there are plenty of such graphs for which this bound is almost tight -- that is we have $e(H) \ge dn - o(n)$). The average degree in any such graph is $2d - o(1)$ and in fact it can happen that all but a few vertices have degree $2d$ (for example, the $d$-power of a path). In order to embed such a spanning graph using the current techniques -- which are mainly based on a `vertex-by-vertex' embedding scheme -- at some point one should connect a vertex to an already embedded neighbourhood of size at least $2d$. In order to find such a vertex we clearly need $n p^{2d} \ge 1$, contributing the term $1/2d$ in the exponent. We remark that in the almost-spanning case a recent result of Conlon and Nenadov \cite{conlon17size} asserts that $p = \tilde\Omega(n^{-1/d})$ suffices. As remarked before, such a bound is optimal up to the logarithmic factor as can easily be seen by calculating the expected number of copies of the $d$-power of the path of length $(1 - \eps)n$.


Finally, we consider the family $\calH(n, \Delta)$ consisting of all graphs on $n$ vertices with maximum degree at most $\Delta$. Recall that \eqref{eq:JKV bound} establishes a threshold for containing a $K_{\Delta+1}$-factor, thus one cannot hope for a universality result in $G_{n,p}$ with $p=o\left(n^{-2/(\Delta + 1)}\right)$. It is a common belief that the bound in \eqref{eq:JKV bound} is actually the correct one but the best known results are far away from it. 
Using an argument based on the ideas of Alon and F\"uredi \cite{AF}, Alon et al. \cite{alon2000universality} showed that $p = \Tilde\Omega(n^{-1/\Delta})$ suffices for a typical $\Gnp$ to be $\calH((1 - \eps)n, \Delta)$-universal. This was subsequently extended to spanning graphs by Dellamonica, Kohayakawa, R\"odl and Ruci\'nski \cite{dellamonica2015improved} ($\Delta \ge 3$) and Kim and Lee \cite{kim2014universality} ($\Delta = 2$). The bound on $p$ has been slightly improved by Ferber, Nenadov and Peter \cite{ferber2015universality} for the subfamily consisting of all such graphs which are not `locally dense' (which, among other graphs, contains forests). 

Note that the value of $p$ in the above mentioned universality results comes naturally -- in this range a typical $\Gnp$ has the property that every subset of  $\Delta$ vertices has a non-empty common neighbourhood. Therefore, at least intuitively, one can expect to find a copy of any graph with maximum degree $\Delta$ using `vertex-by-vertex' embedding. This bound has become a benchmark in the field and much subsequent work on embedding spanning or almost-spanning graphs of maximum degree $\Delta$ in random graphs achieves this threshold. This includes the work of Kohayakawa, R{\"o}dl, Schacht and Szemer{\'e}di \cite{kohayakawa2011sparse} on Ramsey properties of random graphs, the bandwidth theorem for random graphs \cite{allen2015local,bottcher2013almost} and a blow-up lemma by Allen et al. \cite{allen2016blow}. As all these results build on the ideas established for proving the above mentioned universality results, in order to achieve any further progress in these more involved questions we first need to improve our understanding of the universality problem.

To the best of our knowledge, there are only two universality results going beyond $n^{-1/\Delta}$. The first one is of Conlon, Ferber, Nenadov and \v Skori\'c \cite{conlon17almost}, obtaining a bound of $p = \tilde\Omega(n^{-1/(\Delta-1)})$ in the almost spanning case for all $\Delta\geq 3$. Note that this matches \eqref{eq:JKV bound} for $\Delta = 3$, up to the logarithmic factor. The second one is of Ferber, Kronenberg and Luh \cite{ferber2016optimal}, matching \eqref{eq:JKV bound} in the case where $\Delta = 2$ (spanning case). In general, spanning results are known to be much harder to obtain than the corresponding almost spanning version and the problem of breaking the  barrier of $p = \tilde \Omega(n^{-1/\Delta})$ for $\mathcal H(n,\Delta)$-universality remained open for all $\Delta \ge 3$. In our main result we make a first progress in breaking this natural barrier in the spanning case for all $\Delta\ge 3$.

\begin{theorem} \label{thm:spanning}
	Let $\Delta \ge 3$ be an integer and let
	$$
		p \ge \left( n^{-1}\log^3 n\right)^{\frac{1}{\Delta-1/2}}.
	$$
	Then $G_{n,p}$ is w.h.p $\calH(n, \Delta)$-universal.
\end{theorem}

In the proof of Theorem \ref{thm:spanning} we make use of a recent embedding scheme introduced by Conlon and Nenadov \cite{conlon17size} combined with the ideas of Conlon, Ferber, Nenadov and \v Skori\'c \cite{conlon17almost} and the \emph{absorption method} introduced in \cite{rodl2006dirac}. In particular, we follow an idea of Montgomery \cite{montgomery2014embedding} to use \emph{robust} bipartite graphs to build an \emph{absorbing structure} that will make our `finishing part' of the embedding quite simple.

The paper is organised as follows: In the next section we state standard results and introduce some notation. In Section \ref{sec:S_embedding} we prove Lemma \ref{lemma:S_almost_spanning} which is the main building block in the proof of Theorem \ref{thm:spanning_deg} (and contains a key idea for the bounded degree case) and in Section \ref{sec:weaker_ordering} we prove a version of this lemma which is tailored for the proof of Theorem \ref{thm:spanning}. In Section \ref{sec:degenerate} we prove the universality result for the family of $d$-degenerate graphs, and finally, in Section \ref{sec:main} we prove Theorem \ref{thm:spanning}. In the last section we make concluding remarks and give directions for further research.

\input{preliminaries}

\input{almost_spanning_S}

\input{d_degenerate}

\input{bounded_degree}

\input{concluding}



\bibliographystyle{abbrv}
\bibliography{spanning}

\appendix

\input{appendix}

\end{document}

%% file: preliminaries.tex
\section{Notation and preliminaries}

In order to make the arguments and calculations easier to follow, we avoid explicit use of floors and ceilings. Thus, for example, if $S$ is a set and $r \in \mathbb{R}$ we write $|S| = r$ to denote $|S| \in \{\lfloor r \rfloor, \lceil r \rceil\}$. To compensate for any errors caused by rounding we make all inequalities to hold with sufficiently large margin.

Given a graph $H$, we write $xy \in H$ as a shorthand for $\{x,y\} \in E(H)$. The size of the vertex and the edge set of $H$ is denoted by $v(H)$ and $e(H)$, respectively. We use $\Delta(H)$ and $\delta(H)$ to denote the maximum and minimum degree of $H$, respectively. Given a vertex $v \in V(H)$ we use $N_H(v)$ to denote its neighbourhood. Given a subset of vertices $W \subseteq V(H)$ we write $N_H(W) := \bigcup_{w \in W} N_H(w)$ to denote the set of vertices which are adjacent to some vertex in $W$. If $H$ is clear from the context we omit it from the subscript. We say that two vertices in $H$ are at \emph{distance $k$} if a shortest path between them is of length at least $k$ (where the length of a path equals to its number of edges).

Given graphs $G$ and $H$, a mapping $\phi \colon V(H) \rightarrow V(G)$ is said to be an \emph{embedding} of $H$ into $G$, with the notation $\phi \colon H \hookrightarrow G$, if $\phi$ is injective and $\phi(h)\phi(h') \in G$ for every $hh' \in H$. Moreover, $\phi \colon H \hookrightarrow G$ is an \emph{isomorphism} if $\phi(h)\phi(h') \in G$ if and only if $hh' \in H$.

\subsection{Auxiliary bipartite graph $\calB_G(\calL, U)$}
\label{sec:aux_bip}

Suppose we have already embedded a subgraph $H' \subseteq H$ into a host graph $G$. Let $\phi \colon H' \hookrightarrow G$ denote such an embedding and let $U:=V(G)\setminus \phi(V(H'))$ be the set of \emph{unoccupied} vertices of $G$. Moreover, suppose that the set $I := V(H)\setminus V(H')$ of the remaining vertices is independent in $H$. Now, for each $v \in I$ let $L_v = \phi(N_H(v)) \subseteq \phi(V(H'))$ denote the image of $N_H(v)$ in $G$ (note that $N_H(v)$ is already embedded at this stage). 
Observe that a vertex $v \in I$ can be mapped onto a vertex $u \in U$ in an extension of the current embedding $\phi$ only if $L_v \subseteq N_G(u)$ --- in other words, $u$ needs to be adjacent to all the vertices in $L_v$ (see Figure \ref{fig:extend_phi}).

\begin{figure}[h!]	
	\centering
	\includegraphics{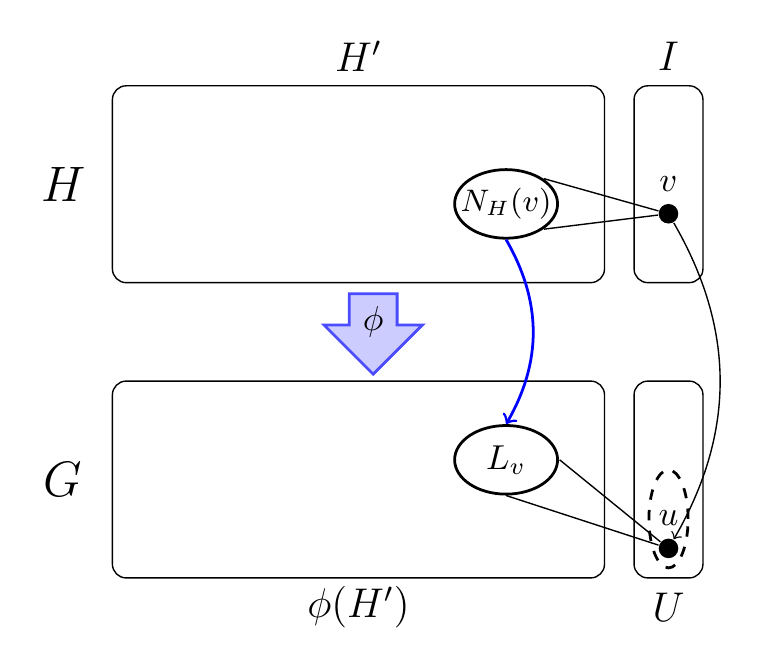}	
	\caption{Dashed subset of vertices of $U$ represent all possible candidates for $v$.}
	\label{fig:extend_phi}
\end{figure}

We define an auxiliary bipartite graph which captures this property:

\begin{definition}[The bipartite graph $\calB_G(\calL, U)$]
Given a graph $G$, a family $\calL$ of subsets of $V(G)$ and a subset $U \subseteq V(G)$, we form the bipartite graph $\calB_G(\calL, U)$ as follows: the vertex set of $\calB(\calL, U)$ consists of $U$ as one part and $\calL$ as the other (that is, each set in $\calL$ represent a single vertex in $\calB_G(\calL, U)$), and the edge set consists of all pairs $L \in \calL$ and $u \in U$ such that $L \subseteq N_G(u)$.
\end{definition}

Observe that whenever $V(H)\setminus V(H')$ is an independent set, an embedding $\phi$ of $H'$ into $G$ can be extended into an embedding of $H$ if and only if $\calB_G(\calL,U)$ contains a \emph{perfect matching}, where $\calL:=\{L_v \colon v\in I\}$ and $U:=V(G)\setminus \phi(V(H'))$. This observation has been used in most of the previous results on embedding spanning graphs. In the proof of Theorem \ref{thm:spanning_deg} the existence of such perfect matching will follow from Hall's criteria (stated in the next section) and expansion properties of random graphs (see Section \ref{sec:expansion_aux}). In the proof of Theorem \ref{thm:spanning} this is significantly more difficult and we resort to the absorbing method.

\subsection{Some results in graph theory}


As remarked in the previous section, we use Hall's criteria in order to finish off an embedding of a desired graph. The following theorem is not the standard version of Hall's Theorem but the equivalence to the original is an easy exercise.

\begin{theorem}[Hall's criteria] \label{thm:hall}
	Let $B = (V_1 \cup V_2, E)$ be a bipartite graph with vertex classes $V_1$ and $V_2$. If $|V_1| = |V_2| = n$ and for every $S \subseteq V_i$ of size $|S| \le n/2$ we have $|N_H(S, V_{3 - i})| \ge |S|$ then $B$ contains a perfect matching.
\end{theorem}

In many places throughout the paper it will be convenient to work with vertices which are sufficiently \emph{independent}, that is, which are far apart. The next easy lemma shows the existence of a large subset of such vertices. As in all our proofs $\Delta$ will be a constant and $S$ a very large set, we did not try to obtain the best possible bound on $S'$.

\begin{lemma} \label{lemma:distance}
	Let $H$ be a graph with maximum degree $\Delta$. For every subset $S \subseteq V(H)$ and $k \in \mathbb{N}$, there exists a subset $S' \subseteq S$ of size $|S'| \ge |S|/\Delta^{k+1}$ such that every two distinct vertices from $S'$ are at distance at least $k$ in $H$.
\end{lemma}
\begin{proof}  Build $S'$ greedily as follows: start
with $X:=S$ and $S':=\emptyset$ and in each step add an arbitrary vertex $v\in X$ to $S'$
and delete the $k$-neighbourhood of $v$ (that is, $\{v\}\cup N_H(v)\cup N^2_H(v)\cup \ldots \cup N^k_H(v)$) from
$X$. Since after each addition of a vertex to $S'$ we delete at most
$$1+\Delta+\Delta(\Delta-1)+\ldots+\Delta(\Delta-1)^{k-1}\leq \Delta^{k+1}$$
vertices from $S$, we conclude $|S'| \ge |S| / \Delta^{k+1}$.
\end{proof}

\subsection{Properties of random graphs}

In the following section we introduce some typical properties of random graphs.

\subsubsection{$F$-matchings} \label{sec:F_matching}

Given graphs $G$ and $F$, we refer to a collection of vertex-disjoint copies of $F$ in $G$ as an \emph{$F$-matching}. This notion naturally generalises the notion of a matching -- which is a set of vertex-disjoint edges (that is, $F$ is an edge)-- to arbitrary structures. The first lemma gives a  bound on $p$ for which a typical $\Gnp$ contains a large $F$-matching for arbitrary $F$.

\begin{lemma}[{\cite[Corollary 3.5]{nenadov17square}}] \label{lemma:factor}
	Let $F$ be a graph. If $p \ge \left( n^{-1}\log^3 n\right)^{\frac{1}{m_1(F)}}$, where
	$$
		m_1(F) = \max \left\{ \frac{e(F')}{v(F') - 1}\colon F'\subseteq F, v(F')\geq 2\right\},
	$$
	then $\Gnp$ w.h.p contains a family of at least $n / 4v(F)$ pairwise vertex-disjoint copies of $F$.
\end{lemma}

The proof of Lemma \ref{lemma:factor} is an easy application of Janson's inequality and can easily be adapted to give an $F$-matching of size $(1 - \eps)n/v(F)$ for any small constant $\eps > 0$. However, showing that one can find an $F$-matching which covers all the vertices of $\Gnp$ is a notoriously difficult problem. This was solved by Johansson, Kahn and Vu \cite{johansson2008factors} for all graphs $F$ which satisfy certain \emph{balancedness} condition.

In some applications (such as the one in Section \ref{sec:delta_absorber}) we are interested in an \emph{anchored} $F$-matching --- an $F$-matching where some of the vertices of each copy of $F$ are already prescribed. More precisely, given graphs $G$ and $F$ and $r$-tuples of vertices $\mathbf{x} \in V(F)^r, \mathbf{y} \in V(G)^r$, for some $1 \le r \le v(F)$, we say that $F' \subseteq G$ is an \emph{$(F, \mathbf{x}, \mathbf{y})$-copy} if there exists an isomorphism $f \colon F \hookrightarrow F'$ such that $f(\mathbf{x}) = \mathbf{y}$. In other words, an $(F, \mathbf{x}, \mathbf{y})$-copy is a copy of $F$ for which the vertices $\mathbf{x}$ are mapped onto $\mathbf{y}$ (a copy of $F$ \emph{anchored} in $\mathbf{y}$). We call vertices $f(F) \setminus \mathbf{y}$ the \emph{internal vertices}.

\begin{definition}
	Let $G$ and $F$ be graphs and $\mathbf{x} \subseteq V(F)^r$ an $r$-tuple of vertices, for some $1 \le r \le v(F)$. Given a family $\mathcal{Y} = \{ \mathbf{y}_i \subseteq V(G)^r \}_{i \in [t]}$ of pairwise disjoint $r$-tuples, we say that a collection $\{F_i \subseteq G\}_{i \in [t]}$ of subgraphs of $G$ forms an \emph{$(F, \mathbf{x}, \mathcal{Y})$-matching} if the following holds:
	\begin{itemize}
		\item $F_i$ is an $(F, \mathbf{x}, \mathbf{y}_i)$-copy for every $i \in [t]$,
		\item $V(F_i) \cap V(F_j) = \emptyset$ for all $i \neq j \in [t]$.	
	\end{itemize}
\end{definition}

The following lemma gives a lower bound on $p$ for which $\Gnp$ admits an $(F, \mathbf{x}, \mathcal{Y})$-matching. In order to state it we need the following version of $m_1$-density: given a graph $F$ and a subset $X \subseteq V(F)$ we denote with $m(F, X)$ the \emph{rooted}-density of $F$, defined as
$$
	m(F, X) = \max_{\substack{F' \subseteq F\\e(F')>0}} \left\{ \frac{e(F')}{v(F') - \max\{1, |V(F') \cap X|\}} \colon
	\text{either } X \subseteq V(F') \text{ or } X \cap V(F') = \emptyset \right\}.
$$

\begin{lemma}[{\cite[lemma 3.3]{nenadov17square}}] \label{lemma:connecting}
	Let $F$ be a graph and $\mathbf{x} \subseteq V(F)^r$ an $r$-tuple of independent vertices in $F$, for some $r \le v(F) - 2$. Given a positive constant $\alpha \in \mathbb{R}$ and a subset $W \subseteq [n]$ of size $|W| \ge \alpha n$, if
	$$
		p \ge \left( n^{-1}\log^3 n \right)^{1/m(F, \mathbf{x})}
	$$
	then $G = \Gnp$ w.h.p has the following property: For every family $\mathcal{Y} = \{\mathbf{y}_i \in (V(G) \setminus W)^r\}_{i \in [t]}$ of $t \le |W| / 4(v(F) - r)$ disjoint $r$-tuples, there exists an $(F, \mathbf{x}, \mathcal{Y})$-matching in $G$ with all internal vertices being in $W$.
\end{lemma}

Similarly as in Lemma \ref{lemma:factor}, the factor $1/4$ in the upper bound on $t$ is somewhat arbitrary and could be replaced by any constant $c < 1$.

\subsubsection{Expansion properties of random graphs}
\label{sec:expansion_aux}


The following two lemmata show certain expansion properties of random graphs or, more precisely, of auxiliary bipartite graphs $\mathcal{B}_G(\calL, U)$ induced by random graphs. The first lemma plays an important role in the proof of one of our main ingredients, Lemma \ref{lemma:S_almost_spanning}. Both lemmata are used to show that for certain $\calL$ and $U$ the corresponding auxiliary bipartite graph contains a perfect matching (utilising Hall's criteria). Proofs of both statements are standard application of Chernoff's inequality thus we omit them (for a similar proof see \cite[Lemma 4.3]{ferber2015universality})

\begin{lemma} \label{lemma:expansion}
	Let $d \in \mathbb{N}$ and $\lambda \in \mathbb{R}$ be a positive constant. Given a subset $U \subseteq [n]$ of size $|U| \ge n / \log n$, if $p \ge \left( n^{-1}\log^2 n \right)^{1/d}$ then $G = \Gnp$ has the following property with probability $1 - O(1/n)$: for every family $\calL \subseteq \binom{V(G)}{d}$ of pairwise disjoint $d$-subsets we have
	$$
		\big| N_\calB(\calL) \big| \ge \begin{cases}
			 |\calL||U| p^d / 2, &\text{ if } |\calL| \le 1 / p^d, \\
			 (1 - \lambda)|U|, &\text { if } |\calL| \ge \log n / p^d,
		\end{cases}
	$$	
	where $\calB = \calB_G(\calL, U)$.
\end{lemma}

\begin{lemma} \label{lemma:expansion_reverse}
	Let $d \in \mathbb{N}$ and $\lambda \in \mathbb{R}$ be a positive constant. Given a family $\calL \subseteq \binom{V(G)}{d}$ of pairwise disjoint $d$-subsets of size $|\calL| \ge n / \log n$, if $p \ge (n^{-1}\log^2 n)^{1/d}$ then $G = \Gnp$ w.h.p has the following property: for every subset $U \subseteq V(G)$ we have
	$$
		\big| N_{\calB}(U) \big| \ge \begin{cases}
			 |U| |\calL| p^d / 2, &\text{ if } |U| \le 1 / p^d, \\
			 (1 - \lambda)|\calL|, &\text { if } |U| \ge \log n / p^d,
		\end{cases}
	$$
	where $\calB = \calB_G(\calL, U)$.
\end{lemma}

%% file: almost_spanning_S.tex
\section{Almost-spanning $S$-embeddings}
\label{sec:S_embedding}

In this section we present one of the main ingredients in the proof of Theorem \ref{thm:spanning_deg} and its refinement tailored to the proof of Theorem \ref{thm:spanning}. In order to motivate its statement we give a brief overview of the strategy used to prove these two theorems.

In the preceding section we introduced the notion of an auxiliary bipartite graph $\calB_G(\calL, U)$ and explained how it comes into play to finish off an embedding of $H$. Briefly, we first embed a subgraph $H' \subseteq H$ obtained from $H$ by removing an independent set of vertices $I$. Then, in order to complete such a partial embedding $\phi$ into an embedding of $H$ we need to argue that there exists a perfect matching in an auxiliary bipartite graph $\calB_G(\calL, U)$, where $\calL = \{\phi(N_H(v)) \colon v \in I\}$ and $U = V(G) \setminus \phi(V(H'))$. This goes smoothly if we are to embed only one graph $H$ as we have a freedom to  \emph{sprinkle}  few edges at the end (that is, to use the standard \emph{multiple exposure} trick) in order to obtain the required matching (and of course, assuming that $p$ is large enough so every subset in $\mathcal L$ will have many common extensions). However, doing so in the universality setting turns out to be more difficult as the error probabilities from the sprinkling parts are way too large for taking a union bound over all possible graphs (and in the proof of Theorem \ref{thm:spanning} we will have another obstacle to pass, namely our edge-probability $p$ is going to be too small for the sprinkling trick to work).
In order to achieve that such a bipartite graph indeed has a perfect matching we need to do some preparations. This is roughly being done as follows: we choose $I$ to be a subset of vertices  of $H$ which are far apart and whose neighbourhoods induce the same graph $F$. Then we first embed the neighbourhoods of these vertices, which corresponds to an $F$-matching. Moreover, we put aside a small subset of vertices $X \subseteq V(G)$ which will help us to verify the Hall's criteria in $\calB_G(\calL, U)$ (note that this refers to the proof of Theorem \ref{thm:spanning_deg}; the proof of Theorem \ref{thm:spanning} is more delicate). Next, we extend the embedding of the neighbourhoods of the vertices from $I$ into an embedding of $H' = H \setminus I$ such that no vertex from $X$ is used. This is accomplished by Lemma \ref{lemma:S_almost_spanning}. Finally, the fact that $X$ is chosen upfront and is not used so far will enable us to prove the existence of a perfect matching in the auxiliary bipartite graph at the end. 

Before we state Lemma \ref{lemma:S_almost_spanning} we need the following definition.

	
\begin{definition}[$S$-embedding] \label{def:s_embedding}
	Given graph $G, H$ and a subset $S \subseteq V(H)$, we say that a mapping $\phi \colon V(H) \rightarrow V(G)$ is an \emph{$S$-embedding}, with the notation $\phi \colon H \hookrightarrow_S G$, if $\phi$ is injective and $\phi(h)\phi(h') \in G$ for every $hh' \in H \setminus E(H[S])$.
\end{definition}

This definition will become clearer after the statement of Lemma \ref{lemma:S_almost_spanning}. The following simple fact is used throughout the proofs of Theorem \ref{thm:spanning_deg} and Theorem \ref{thm:spanning}: Suppose $G_1$ and $G_2$ are graphs on the same vertex set. If $\phi' \colon H[S] \hookrightarrow G_1$ and $\phi \colon H \hookrightarrow_S G_2$ extends $\phi'$, then $\phi \colon H \hookrightarrow G_1 \cup G_2$.

Now we are ready to state our main embedding lemma. We remark that the proof is entirely based on an embedding scheme introduced by Conlon and Nenadov \cite{conlon17size}.

\begin{lemma} \label{lemma:S_almost_spanning}
	Let $d, \Delta \in \mathbb{N}$ be such that $1 \le d \le \Delta$ and $\alpha, \gamma \in \mathbb{R}$ positive constants. Given a subset $W \subseteq [n]$ of size $|W| \ge \alpha n$, if
	$$
		p \ge \left( n^{-1}\log^3 n\right)^{1/d}
	$$
	then $G = \Gnp$ w.h.p has the following property: For every subset $X \subseteq V(G) \setminus W$, every graph $H \in \calH(n - |X| - \gamma n, \Delta)$ and every subset $S \subseteq V(H)$ such that there exists an ordering $(h_1, \ldots, h_m)$ of $V(H) \setminus S$ with
	$$
		|N_H(h_i, S \cup \{h_1, \ldots, h_{i-1}\})| \le d \quad \text{ for every } 1 \le i \le m,
	$$
	any injective mapping $\phi' \colon S \rightarrow V(G) \setminus (W \cup X)$ can be extended to an $S$-embedding $\phi \colon H \hookrightarrow_S G \setminus X$.
\end{lemma}

Before proving Lemma \ref{lemma:S_almost_spanning} we briefly spell out its statement: Suppose we are given a small set $W \subseteq V(G)$. Then the lemma says that for \emph{any} set $X \subseteq V(G) \setminus W$, \emph{any} graph $H$ such that $G$ is large enough to accommodate it without using $X$, \emph{any} injective mapping which avoids $W$ and $X$ can be extended to an $S$-embedding of the whole graph $H$ which avoids $X$. Here one should think of $H$ as $H'$ from the preceding discussion and $S$ as being the neighbourhood of $I$.

The role of $X$ has already been explained. The main role of the set $W$ in Lemma \ref{lemma:S_almost_spanning} is to prevent the embedding process from getting stuck. For the convenience of the reader we demonstrate it on the following example: suppose $V(H) = S \cup V'$ and assume there exists an edge $sw \in H$ such that $s \in S$ and $w \in V'$. Because we allow for an arbitrary injection $\phi \colon S \rightarrow V(H) \setminus (W \cup X)$, it could happen that all the neighbours of $s$ outside of $X$ and $W$ are in $\phi(S)$, that is, $N_G(s) \setminus (X \setminus W) \subseteq \phi(S)$. If it was not for the set $W$, this would prevent us from completing the embedding. However, having the set $W$ put aside (a set which has a `typical' behavior), we expect that $s$ has sufficiently large neighbourhood into $W$ and, as $\phi(S) \cap W = \emptyset$, we expect to find a candidate for $w$ in $W$.


We note that something along these lines was used, for example, in \cite{alon2007embedding,dellamonica2015improved,ferber2015universality} where in order to embed the next batch of vertices one uses a set which was put aside especially for that purpose.

\begin{proof}[Proof of Lemma \ref{lemma:S_almost_spanning}]
Set $k = 2 \log n/\log \log n$ and let $W_1, \ldots, W_k \subseteq W$ be disjoint subsets, each of size $|W_i| = n / \log n$. Then $G = \Gnp$ w.h.p satisfies the property of Lemma \ref{lemma:expansion} with $\lambda = \gamma/4$ for every $W_i$ and $V(G)$ (as $U$). We show that such $G$ satisfies the property of the lemma.

Consider a subset $X \subseteq V(G) \setminus W$ and set $W_0 := V(G) \setminus (W_1 \cup \ldots \cup W_k \cup X)$. Let $H \in \calH(n - |X| - \gamma n, \Delta)$ and $S \subseteq V(H)$ be such that there exists an ordering $(h_1, \ldots, h_m)$ of $V(H) \setminus S$ with
$$
	|N_H(h_i, S \cup \{h_1, \ldots, h_{i-1}\})| \le d,
$$
for every $1 \le i \le m$. Given an arbitrary injection $\phi' \colon V(G) \setminus (W \cup X)$ we construct an extension $\phi \colon H \hookrightarrow_S G \setminus X$ of $\phi'$ by iteratively defining $\phi(h_i)$ for $i = 1, \ldots, m$ as follows:
\begin{enumerate}[(i)]
	\item If $L_i := \phi(N_H(h_i, S \cup \{h_1, \ldots, h_{i-1}\}))$ is empty set then choose an arbitrary $v_i \in W_0 \setminus (\phi(S) \cup \{v_1, \ldots, v_{i-1}\})$ and set $\phi(h_i) := v_i$;
	\item Otherwise, for each $j \in \{0, \ldots, k\}$ set
	$$
		C_i^j :=  \{w \in W_{j} \setminus (\phi(S) \cup \{v_1, \ldots, v_{i-1}\}) \colon L_i \subseteq N_G(w) \}
	$$
	and let $j_i \in \{0, \ldots, k\}$ be the smallest index such that $C_i^{j_i}$ is non-empty. Choose arbitrary $v_i \in C_i^{j_i}$ and set $\phi(h_i) := v_i$.
\end{enumerate}
In other words, each $h_i$ is mapped into the first `free' set $W_j$. Observe that
\begin{equation} \label{eq:H_size_W_0}
	v(H) \le n - |X| -  \gamma n < |W_0| - \gamma n / 2,
\end{equation}
which immediately implies (i) is well-defined. Assuming that (ii) can also be always performed, which we show next, definitions of $L_i$ and $C_i^{j_i}$ imply that $\phi$ is an $S$-embedding of $H$ into $G \setminus X$, which concludes the proof. It will be convenient to assume that in case (ii) cannot be performed in some step $i$, we just choose $v_i$ to be an arbitrary `free' vertex in $W_0$ and proceed to the next step (\eqref{eq:H_size_W_0} shows this can always be done).

Let $J_j := \{i \in [m] \colon \phi(h_i) \in W_j\}$ denote the set of indices of vertices which are mapped into $W_j$, for $j \in \{1, \ldots, k\}$. In order to prove that the step (ii) is possible it suffices to show
\begin{equation} \label{eq:W_j_invariant}
	\left| J_j \right| \le \frac{ 2^{j-1} (d\Delta^3 \log n)^j}{n^{j-1} p^{j d}}
\end{equation}
for every $j \in \{1, \ldots, k\}$. Indeed, assuming this is true from the choice of $k$ we have $|J_k| = 0$. Furthermore, from the assumption that $G$ satisfies the property of Lemma \ref{lemma:expansion} for $W_k$ (as $U$) we have that every subset of at most $d$ vertices has a common neighbour in $W_k$. Finally, as $W_k \cap \phi(S) = \emptyset$ and $|L_i| \le d$ the previous two observations imply $C_i^k \neq \emptyset$ in every step of the process, thus (ii) is always well-defined. It remains to show \eqref{eq:W_j_invariant}.

Let us first consider the case $j = 1$. Note that $|L_i| \ge 1$ for every $i \in J_1$  as otherwise $h_i$ is mapped into $W_0$. Using the pigeon-hole principle and Lemma \ref{lemma:distance} with $k = 2$, there exists $D \in \{1, \ldots, d\}$ and a subset $J' \subseteq J_1$ of size $|J'| \ge |J_1| / d\Delta^3$ such that:
\begin{enumerate}[(a)]
	\item $|L_i| = D$ for every $i \in J'$, and
	\item $h_i$ and $h_{i'}$ do not have a common neighbour, for every $i \neq i' \in J'$.
\end{enumerate}
In particular, (b) implies $L_i \cap L_{i'} = \emptyset$. Therefore, if $|J'| \ge \log n / p^d$ we can apply the property of Lemma \ref{lemma:expansion} with $\calL = \{L_i\}_{i \in J'}$ and $V(G)$ (as $U$) to deduce that all but at most $\gamma n / 4$ vertices $v \in V(G)$ satisfy $L_i \subseteq N_G(v)$ for some $i \in J'$. Moreover, from \eqref{eq:H_size_W_0} we have
$$
	|W_0 \setminus \phi(H)| = |W_0| - v(H) \ge \gamma n / 2
$$
thus there exists a vertex $v \in W_0 \setminus \phi(H)$ and $i \in J'$ such that $\phi(L_i) \subseteq N_G(v)$. In other words, $C_i^0 \neq \empty$ which contradicts the assumption that $h_i$ was mapped into $W_1$. This concludes $|J'| < \log n / p^d$ and, consequently, $|J_1| \le d\Delta^3 \log n / p^d$.

We apply similar argument to conclude \eqref{eq:W_j_invariant} for all $j \in \{2, \ldots, k\}$. Let us assume, towards the contradiction, that there exists $j \in \{2, \ldots, k\}$ for which \eqref{eq:W_j_invariant} is not satisfied and consider the smallest such $j$. Then there exists a subset $J' \subseteq J_j$ of size
\begin{equation} \label{eq:J_prim}
	|J'| \ge |J| / d\Delta^3 \ge \frac{ 2^{j-1} (d\Delta^3)^{j-1} \log^{j} n}{n^{j-1} p^{j d}}
\end{equation}
and $D \in \{1, \ldots, d\}$ such that:
\begin{enumerate}[(a)]
	\item $|L_i| = D$ for every $i \in J'$, and
	\item $h_i$ and $h_{i'}$ do not have a common neighbour, for every $i \neq i' \in J'$.
\end{enumerate}
Without loss of generality we may assume $J'$ is exactly of the size indicated on the right hand side in \eqref{eq:J_prim}, which is easily seen to be $o(1/p^d)$. If this is not the case, then we simply consider a subset of $J'$ of that size. Therefore, from the property of Lemma \ref{lemma:expansion} for $W_{j-1}$ (as $U$)  there exist at least
$$
	|J'||W_{j-1}|p^d / 2 \ge \frac{2^{j-2} (d \Delta^3)^{j-1} \log^j n}{ n^{j-1}p^{jd}} \cdot \frac{n}{\log n}p^d > \frac{2^{j-1} (d \Delta^3 \log n)^{j-1} }{n^{j-2} p^{(j-1)d}}
$$
vertices $w \in W_{j-1}$ such that $L_i \subseteq N_G(w)$ for some $i \in J'$. On the other hand, from the assumption that $j \ge 2$ is the smallest index for which \eqref{eq:W_j_invariant} fails we have
$$
	|J_{j-1}| \le \frac{2^{j-1} (d\Delta^3 \log n)^{j-1}}{n^{j-2} p^{(j-1)d}}
$$
thus there exists at least one `free' vertex $w \in W_{j-1} \setminus  J_{j-1} = W_{j-1} \setminus \phi(H)$ such that $L_i \subseteq N_G(w)$ for some $i \in J'$. This implies $C_i^{j-1} \neq \emptyset$ which finally contradicts the assumption that $j$ is the smallest index for which $C_i^j \neq \emptyset$.
\end{proof}

\subsection{Weaker ordering condition for $d = \Delta$}
\label{sec:weaker_ordering}

Note that any graph with maximum degree $\Delta$ is also $\Delta$-degenerate and there are many such graphs which are not $(\Delta-1)$-degenerate (for example, every $\Delta$-regular graph). Therefore, applying Lemma \ref{lemma:S_almost_spanning} on such graphs necessarily requires $p = \tilde \Omega(n^{-1/\Delta})$ even if $S = \emptyset$, which is exactly the bound we try to overcome in Theorem \ref{thm:spanning}. Luckily, $\Delta$-regular graphs are `close' to be $(\Delta-1)$-degenerate: for example, it is a simple exercise to show that by removing a vertex or an edge from each connected component one obtains a $(\Delta - 1)$-degenerate graph. This observation enables us to achieve slightly better bound on $p$ than given by Lemma \ref{lemma:S_almost_spanning}, which is stated in the following lemma. 


\begin{lemma} \label{lemma:delta_S_embedding}
	Let $\Delta \ge 2$ be an integer and $\alpha, \gamma \in \mathbb{R}$ positive constants. Given a subset $W \subseteq [n]$ of size $|W| \ge \alpha n$, if
	$$
		p \ge \left( n^{-1}\log^3 n \right)^{1/(\Delta - 1/2)}
	$$
	then $G = \Gnp$ w.h.p has the following property: For every $X \subseteq V(G) \setminus W$, every graph $H \in \calH(n - |X| - \gamma n, \Delta)$ and every subset $S \subseteq V(H)$ such that
	$$
		|N_H(h, S)| \le \Delta - 1 \quad \text{ for every } h \in V(H) \setminus S,
	$$ 	
	any injective mapping $\phi' \colon S \rightarrow V(G) \setminus (W\cup X)$ can be extended to an $S$-embedding $\phi \colon H \hookrightarrow_S G \setminus X$.
\end{lemma}

We use the following strategy in the proof of Lemma \ref{lemma:delta_S_embedding}: First remove a carefully chosen matching from $H \setminus S$ such that there exist an ordering $(h_1, \ldots, h_m)$ of the remaining vertices satisfying the condition of Lemma \ref{lemma:S_almost_spanning} with $d = \Delta - 1$. Embed these vertices using Lemma \ref{lemma:S_almost_spanning} and put back the matching using Lemma \ref{lemma:expansion_edges} stated below. Note that the lower bound on $p$ is the best possible (up to the logarithmic factor) given the condition $|N_H(h, S)| \le \Delta - 1$: if there exists an edge $hh' \in H \setminus S$ such that both $h$ and $h'$ have $\Delta - 1$ neighbours in $S$, then we need $p = \tilde \Omega(n^{-1/(\Delta - 1/2)})$ just to embed this one edge.

While, in principle, we could have stated more complicated conditions which would allow one to obtain better bounds on $p$, we opted not to do so for following reasons: (i) this is the simplest statement which allows us to go below $n^{-1/\Delta}$ in Theorem \ref{thm:spanning}, (ii) any more complicated condition would further obscure the preparation of a graph $H$ in the proof of Theorem \ref{thm:spanning}, and (iii) this approach cannot improve $p$ past $n^{-1/(\Delta - 1)}$. While the bound of $n^{-1/(\Delta - 1)}$ could be achieved by requiring that each connected component of $H \setminus S$ contains a cycle (see the proof of the main result from \cite{conlon17almost}), attaining such a condition in the proof of Theorem \ref{thm:spanning} would be difficult. On top of it, there are other places in the proof where going below $n^{-1/(\Delta - 1/2)}$ would require new ideas. More on this will be said in the last section.

As remarked before, the following lemma helps us to finish off the embedding in the proof of Lemma \ref{lemma:delta_S_embedding}. 
The proof combines Lemma \ref{lemma:connecting} and a standard application of Janson's inequality. As it is somewhat technical and does not introduce new ideas, we leave it to the appendix.

\begin{lemma} \label{lemma:expansion_edges}
Let $\Delta \ge 2$ be an integer and $\alpha \in \mathbb{R}$ positive constants. Given a subset $W \subseteq [n]$ of size $|W| \ge \alpha n$, if
$$
	p \ge \left(n^{-1}\log^3 n \right)^{1/(\Delta - 1/2)}
$$
then $G = \Gnp$ w.h.p has the following property: For every subset $W' \subseteq [n] \setminus W$ and every family $\{(A_i, B_i)\}_{i \in [t]}$ of pairs of subsets $A_i, B_i \subseteq [n] \setminus (W \cup W')$ of size $|A_i| = |B_i| = \Delta - 1$ such that
	\begin{itemize}
		\item $2 t \le |W'|$ and
		\item no vertex of $G$ appears in more than $\Delta$ pairs,
	\end{itemize}
there exists a family of vertex-disjoint edges $\{x_i y_i \in G[W \cup W']\}_{i \in [t]}$ such that $A_i \subseteq N_G(x_i)$ and $B_i \subseteq N_G(y_i)$ for every $i \in [t]$.
\end{lemma}

With Lemma \ref{lemma:expansion_edges} at hand, we finish the proof of Lemma \ref{lemma:delta_S_embedding}.

\begin{proof}[Proof of Lemma \ref{lemma:delta_S_embedding}]
	Let $W \subseteq [n]$ be a subset of size $|W| \ge \alpha n$. Without loss of generality we may assume $\alpha \le \gamma$, as otherwise we can simply take a subset of $W$ of size $\gamma n$. Let $W_0, W_1 \subseteq W$ be disjoint subsets of size $|W_i| = \alpha n / 2 \le \gamma n / 2$. Then $G = \Gnp$ w.h.p has the following properties:
	\begin{itemize}
		\item the property of Lemma \ref{lemma:S_almost_spanning} with $\Delta - 1$ as $d$ and $W_0$ as $W$,
		\item the property of Lemma \ref{lemma:expansion_edges} with $W_1$ (as $W$).	
	\end{itemize}
	We show that these properties imply that $G$ satisfies the property of the lemma. To this end, consider a subset $X \subseteq V(G) \setminus W$, a graph $H \in \calH(n - |X| - \gamma n, \Delta)$ and a subset $S \subseteq V(H)$ which satisfy conditions of the lemma.


	\medskip
	\noindent
	\textbf{Prepare $H$.} 
	Let $\{H_i \subseteq H \setminus S\}_{i \in \calI}$ be the family of all connected components in $H \setminus S$ with the property that $\deg_H(w) = \Delta$ for all $w \in H_i$. From each $H_i$ choose one arbitrary edge $a_i b_i\in H_i$ (this is possible since $|N_H(w, S)| \le \Delta - 1$ for all $w \in H \setminus S$) and set $M = \bigcup_{i \in \calI} \{a_i, b_i\}$ and $H' = H \setminus M$. We now show that there exists an ordering $(h_1, \ldots, h_m)$ of $V(H') \setminus S$ such that
	\begin{equation} \label{eq:degree_satisfied}
		|N_H(h_i, S \cup \{h_1, \ldots, h_{i-1}\})| \le \Delta - 1
	\end{equation}
	for every $1 \le i \le m$. Indeed, let $D_0 \subseteq V(H') \setminus S$ denote the set of all vertices having a neighbour in $M$ (in particular, they all have degree at most $\Delta-1$ in $H\setminus M$), and inductively define $D_j = N_{H'}(D_{j-1}) \setminus (S \cup \bigcup_{j' < j} D_{j'})$ for every $j \ge 1$. Clearly, if $D_j = \emptyset$ for some $j$ then $D_{j+k} = \emptyset$ for every $k \ge 1$. Moreover, as each connected component of $H' \setminus S$ contains a vertex of degree at most $\Delta - 1$ in $H'$ (that is, it contains a vertex from $D_0$) we have $V(H') \setminus S = \bigcup_{j \ge 0} D_j$. In addition, since the $D_j$'s are disjoint, it follows that there exists a smallest $\ell \ge 0$ such that $D_\ell = \emptyset$, and therefore, $V(H') \setminus S = \bigcup_{j = 0}^{\ell-1} D_j$. Now, observe that the ordering $(D_{\ell-1}, D_{\ell-2}, \ldots, D_0)$ of $V(H') \setminus S$ has the property that each vertex which is not in $D_0$, has a neighbour `to the right' and therefore at most $\Delta - 1$ neighbours `to the left' (even if we add $S$ at the beginning of the ordering). As vertices in $D_0$ have degree at most $\Delta - 1$ in $H\setminus M$, by arbitrarily ordering vertices within each $D_j$, each vertex of $D_0$ has at most $\Delta-1$ neighbours `to the left' (again, even if we add $S$ at the beginning of the ordering). Therefore, we obtain a desired ordering of $V(H') \setminus S$.

	\medskip
	\noindent
	\textbf{Embed $H$.}
	Consider an arbitrary injective mapping $\phi' \colon S \rightarrow V(G) \setminus W$. From $v(H') + |M| \le v(H) \le n - |X| - \gamma n$ we have
	\begin{equation} \label{eq:W_1_bound}
		v(H') \le n - (|X| + |W_1|) - \gamma n / 2.
	\end{equation}
	Owing to \eqref{eq:degree_satisfied}, we can apply Lemma \ref{lemma:S_almost_spanning} with $X \cup W_1$ (as $X$) and $W_0$ (as $W$) to obtain an $S$-embedding $\phi \colon H' \hookrightarrow_S G \setminus (X \cup W_1)$ which extends $\phi'$. Let $W' := V(G) \setminus (X \cup W_1 \cup \phi(H'))$ be the set of `unused' vertices. For each $i \in \calI$  set $A_i := \phi(N_H(a_i) \setminus b_i)$ and $B_i := \phi(N_H(b_i) \setminus a_i)$. Observe that $|A_i| = |B_i| = \Delta - 1$ as otherwise we would not remove the edge $a_i b_i$ in the first place. Moreover, as
	$$
		|M| \le n - |X| - \gamma n - v(H') = |W'| + |W_1| - \gamma n \le |W'|
	$$
	we conclude $2 |\calI| = |M| < |W'| - \gamma n / 2$. Therefore, we can apply Lemma \ref{lemma:expansion_edges} to obtain a family $\{x_i y_i \in G[W' \cup W_1]\}_{i \in \calI}$ such that
	$$
		\phi(N_H(a_i) \setminus \{b_i\}) \subseteq N_G(x_i) \qquad \text{ and } \qquad \phi(N_H(b_i) \setminus \{a_i\}) \subseteq N_G(y_i)
	$$
	for every $i \in \calI$. By setting $\phi(a_i) := x_i$ and $\phi(b_i) := y_i$ for each $i \in \calI$ we obtain a desired $S$-embedding $\phi \colon H \hookrightarrow_S G \setminus X$.
\end{proof}

%% file: d_degenerate.tex
\section{Universality for $d$-degenerate graphs}
\label{sec:degenerate}

In this section we prove Theorem \ref{thm:spanning_deg}. The proof demonstrates some of the main ideas and serves as a warm up towards the more difficult proof of Theorem \ref{thm:spanning}.

A brief overview of the proof strategy was already given at the beginning of Section \ref{sec:S_embedding}. We now give a more detailed description on how to embed one particular $d$-degenerate graph $H$ and then use this approach to show the desired universality result.

\begin{enumerate}
	\item First, we choose a subset of vertices $D \subseteq V(H)$ of degree at most $2d$ which will be embedded at the end. As the average degree of every $d$-degenerate graph is at most $2d$ and the maximum degree is bounded by a constant $\Delta$, we have many choices for such vertices. In particular, we choose $D$ in such a way that every two vertices in it are far apart (note that $D$ is an independent set). Next, for each $w \in D$ choose a small subset $S_w \subset V(H)$ such that $S_w$ contains the vertex $w$ and its neighbourhood and no vertex outside of $S_w$ sends more than $d$ edges into $S_w$. Importantly, the $S_w$'s are chosen in such a way that each induced graph $H[S_w]$ is isomorphic to the same (connected) graph $F^*$ and each $w$ has the same role in $F^*$ (say, it has the role of a vertex $z^* \in F^*$). By virtue of being connected, we have that each vertex in $S_w$ is `close' to $w$ (as $S_w$ is small) and, since every two vertices in $D$ are far apart, $S_w$'s are pairwise disjoint.

	\item Embed $|D|$ copies of $F = F^* \setminus z^*$ into $G$ using Lemma \ref{lemma:factor}. Note that we can associate each such copy with an induced graph $H[S_w \setminus \{w\}]$, for some $w \in D$ (see Figure \ref{fig:phase1}). In other words, we embed a subgraph $H[S]$ where $S := \bigcup_{w \in D} S_w \setminus \{w\}$.  Furthermore, we put aside a subset of vertices $X \subseteq V(G)$ which are not being used by these copies. Together with the fact that at this point all neighbours of the vertices from $D$ are embedded, this will help us to finish off an embedding.

	\begin{figure}[h!]
		\centering
		\subfloat[First phase (step 2.)]{\raisebox{0.48cm}{\includegraphics[scale=0.9]{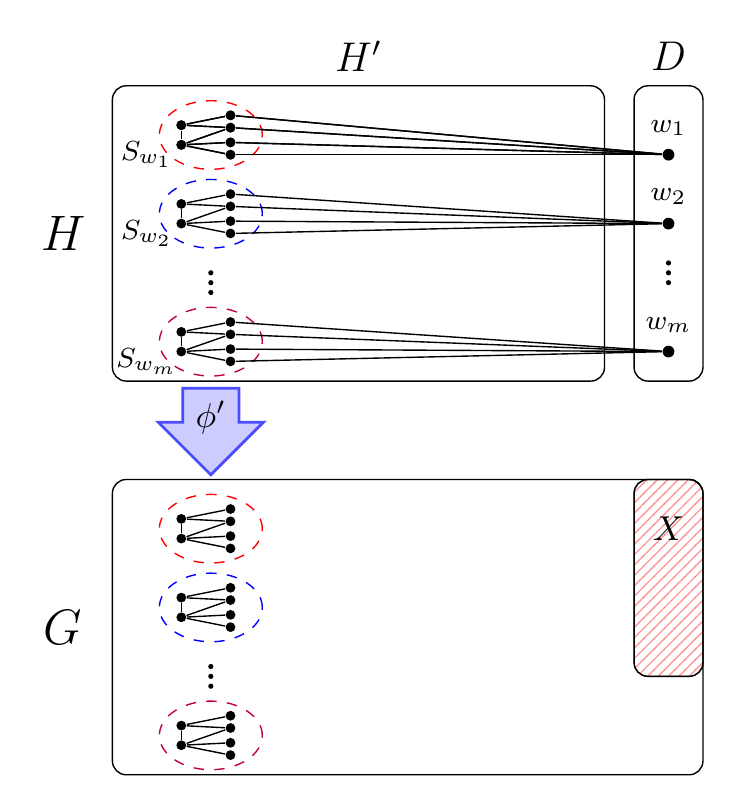}} \label{fig:phase1}}
		\hspace{1.5cm}
		\subfloat[Second phase (step 3.)]{\includegraphics[scale=0.9]{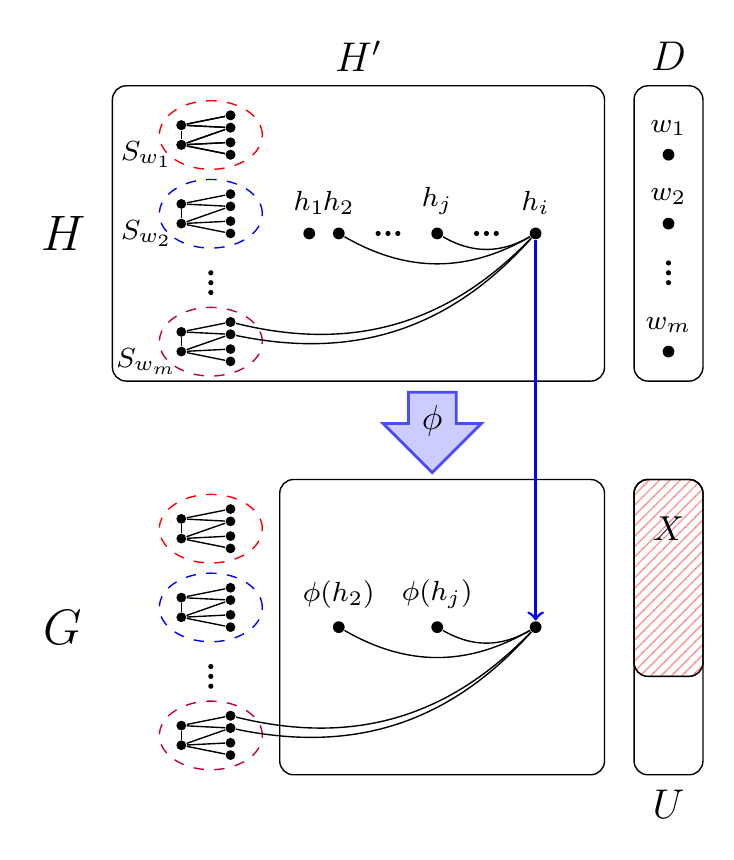} \label{fig:phase2}}
		\caption{Embedding a $2$-degenerate graph $H$.}
	\end{figure}

	\item Next, we wish to extend an embedding of $H[S]$ into an embedding $\phi$ of $H' := H \setminus D$ while avoiding the set $X$ (see Figure \ref{fig:phase2}). This is done by using Lemma \ref{lemma:S_almost_spanning} with $2d$ (as $d$): Since vertices in $D$ are far apart each vertex outside of $S$ can be adjacent to at most one set $S_w$. Therefore, owing to the property that no vertex outside of $S$ sends more than $d$ edges into any $S_w$ we conclude that no vertex sends more than $d$ edges into $S$. This is the main reason why we could not choose $S_w$ to be only the neighbourhood of $w$. As $H' \setminus S$ is $d$-degenerate we conclude there exists an ordering which satisfies the assumption of Lemma \ref{lemma:S_almost_spanning} with $2d$ (as $d$).

	\item Finally, we use the expansion properties described in Section \ref{sec:expansion_aux} to show that $\calB(\calL, U)$ has a perfect matching, where $\calL := \{\phi'(N_H(w)) \colon w \in D\}$ and $U$ is the set of unoccupied vertices in $G$. Consequently, this will imply the existence of an embedding of $H$ into $G$ (see Section \ref{sec:aux_bip}). The fact that $\calL$ is determined in the first phase of the embedding procedure will come in very handy for proving the universality result.

	Note that here we heavily rely on the fact that each $N_H(w)$ is of size at most $2d$, thus we can apply the lemmata from Section \ref{sec:expansion_aux} with $p = \tilde \Omega(n^{-1/2d})$. Finally, note that the role of the set $X$ here is similar to the role of the set $W$ in the statement of Lemma \ref{lemma:S_almost_spanning}: as we have no control on how the subgraph $H' \setminus S$ is being embedded, it could happen that all the vertices which are candidates for some $w \in D$ are already used. If this happens then we have no chance to finish our embedding. Having the set $X$ fixed in advance guarantees that this will not be the case.
\end{enumerate}

We remark that the main difference between the proofs of Theorem \ref{thm:spanning_deg} and Theorem \ref{thm:spanning} lies in Step 4 (which, consequently, makes other steps more difficult as well). In particular, if $H$ is a $\Delta$-regular graph then we necessarily have $|N_H(w)| = \Delta$ and in order to use expansion properties from Section \ref{sec:expansion_aux} we need $p = \tilde \Omega(n^{-1/\Delta})$ which is exactly the bound we aim to break. 


The following technical definition captures the main properties of subsets $S_w$ described in step 1.

\begin{definition} \label{def:degenerate}
	Let $d, K\in \mathbb{N}$, let $F^*$ be a graph and $z^* \in V(F^*)$. Define $\calD_d(F^*, z^*,K)$ to be the family of all graphs $H$ for which the following holds: there exist a set $D \subseteq V(H)$ of size at least $|V(H)| / K$ and a family of subsets $\{S_w\}_{w \in D}$ (where each $S_w\subseteq V(H)$) such that the following is true,
	\begin{enumerate}[(D1)]
		\item for each $w \in D$ we have $\{w\} \cup N_H(w) \subseteq S_w$ and $|N_H(w)| \le 2d$ and
		\item there exists an isomorphism $f_w \colon H[S_w] \hookrightarrow F^*$ which maps $w$ to $z^*$;
		\item $S_w \cap S_{w'} = \emptyset$ and there are no edges between $S_w$ and $S_{w'}$, for every $w \neq w' \in D$;		
		\item for each vertex $w \in V(H) \setminus \bigcup_{w \in D} S_w$ we have $|N_H(w) \cap \bigcup_{w \in D} S_w| \le d$.
	\end{enumerate}
\end{definition}

In the following lemma we show that for every $d$-degenerate graph $H$ there exists a small graph $F^*$, a relatively small constant $K$ and $z^*\in V(F^*)$ such that $H\in \calD_d(F^*,z^*,K)$.

\begin{lemma} \label{lemma:F_star_lemma}
	Let $d, \Delta \in \mathbb{N}$. Then there exists $K = K(d, \Delta)$ such that for every $d$-degenerate graph $H$ with maximum degree at most $\Delta$ there exists a $d$-degenerate graph $F^*$ with at most $5d^2$ vertices and a vertex $z^* \in F^*$ such that $H \in \calD_d(F^*, z^*,K)$.
\end{lemma}

\begin{proof}
In the calculations used throughout the proof we shall often use very generous estimates as they only influence the constant $K$ which is not very important for our purposes.

Let $n = v(H)$ denote the number of vertices of $H$. First, we show that $H$ contains at least $n/3d$ many vertices of degree at most $2d$. Let $W$ denote the set of all such vertices. Clearly,
$$
	2e(H)=\sum_{v\in W}\deg_H(v)+\sum_{v\notin W}\deg_H(v)\geq (2d+1)(n-|W|).
$$
On the other hand, as $H$ is $d$-degenerate it has at most $d n$ many edges. Combining those two bounds we obtain
$$2d n \geq  (2d+1)(n-|W|),$$
which, after rearranging, gives
$$|W| \geq n/(2d+1) \ge n/3d.$$
Second, let $k=20d^2$ and let $D'\subseteq W$ be a largest subset of vertices such that every two are of distance at least $k$. Lemma \ref{lemma:distance} guarantees that $D'$ is of size
$$
	|D'| \ge |W| / \Delta^{k+1} \geq \frac{n}{3 d \Delta^{k+1}}.
$$

Next, for each vertex $w\in D'$ we define a graph $S_w$ as follows:
Start with $S_w:=\{w\}\cup N_H(w)$ and as long as there exists a vertex $v\in V(H)$ with $|N_H(v) \cap S_w | \geq d+1$ pick such a vertex and update $S_w:=S_w\cup\{v\}$. Note that after $t$ steps the set $S_w$ is of size $t+ \deg_H(v)+1$ and $H[S_w]$ contains at least $\deg_H(v) + (d+1)t$ edges. Moreover, as every subgraph of a $d$-degenerate graph is clearly $d$-degenerate, it follows that $H[S_w]$ contains at most $d|S_w|=d(t+\deg_H(v)+1)$ edges. Thus from
$$
 	\deg_H(v) + (d+1) t \le d (t + \deg_H(v) + 1)
$$
and $\deg_H(v) \le 2d$ we conclude $t \le 2d^2$. In other words, the above process  terminates after at most $2d^2$ steps thus each $S_w$ is of size at most $2d^2+\deg_H(v)+1 \leq 3d^2$ (which holds for every $d\geq 3$). Moreover, by the construction we have that each $S_w$ is connected. Therefore, we conclude that for every $w\neq w'$ we have
 $S_w\cap S_w'\neq \emptyset$ as otherwise there exists a path of length at most $|S_w|+|S_{w'}| \le 6d^2 < k$ between $w$ and $w'$, contradicting $w, w' \in D'$. Similarly, we conclude that there is no edge between $V(S_w)$ and $V(S_{w'})$ for $w \neq w' \in D'$.

Finally we are ready to define the desired $D\subseteq D'$. Note that there are at most $c := 2^{9d^4} 3d^2$ different labelled (and rooted) graphs on at most $3d^2$ vertices. As $|S_w| \le 3 d^2$ for every $w \in D'$, by the pigeon-hole principle there exists a subset $D \subseteq D'$ of size
$$
	|D| \ge |D'|/ c \ge n / (3d\Delta^{k+1} c) = n / K
$$
such that all the graphs $S_w$ (for $w \in D$) are exactly the same and have the same role of the vertex $w$.

In order to complete the proof it remains to show that properties $(D1)-(D4)$ hold. Properties $(D1)-(D3)$ are clear from the construction. Regarding $(D4)$, note that if a vertex $v \in V(H)\setminus \left(\bigcup_{w\in D}S_w\right)$ has neighbours in  two distinct $S_w,S_{w'}$, then there exists a path of length at most $6d^2+1<k$ between $w$ and $w'$, which contradicts the definition of $D'$. Therefore, for every vertex $v\in V(H)\setminus \left(\bigcup_{w\in D}S_w\right)$ we have that $N_H(v)$ intersects at most one $S_w$ and by the construction of $S_w$ we have that this intersection is of size at most $d$. This completes the proof.
\end{proof}

The next lemma is the heart of the matter. It states that w.h.p a random graph $G = \Gnp$ contains all graphs in $\mathcal H(n,\Delta,d)$ which are in the same family $\calD_d(F^*,z^*,K)$ (where $F^*,z^*$ and $K$ are fixed). 

\begin{lemma} \label{lemma:univ_deg_F}
	Let $d, \Delta \in \mathbb{N}$ be such that $d \le \Delta$. Given a positive constant $K$, a graph $F^*$ with at most $3d^2$ vertices and $z^* \in V(F^*)$, if $p \ge (n^{-1}\log^3 n)^{\frac{1}{2d}}$ then $\Gnp$ is w.h.p universal for the family of graphs $\calH(n, \Delta, d) \cap \calD_d(F^*, z^*, K)$.
\end{lemma}

Using Lemma \ref{lemma:F_star_lemma} and Lemma \ref{lemma:univ_deg_F} we easily finish the proof of Theorem \ref{thm:spanning_deg}.

\begin{proof}[Proof of Theorem \ref{thm:spanning_deg}]
	Let $K$ be a constant given by Lemma \ref{lemma:F_star_lemma}. It follows from Lemma \ref{lemma:univ_deg_F} and the union bound that $G = \Gnp$ is w.h.p universal for the family of graphs
	$$
		\calH(n, \Delta, d) \cap \left[ \bigcup_{\substack{F^* \\ z^* \in F^*}} \calD_d(F^*, z^*, K) \right],
	$$
	where the union bound goes over all pairs $(F^*, z^*)$ where $F^*$ is a $d$-degenerate graph with at most $3d^2$ vertices and $z^* \in V(F^*)$. From Lemma \ref{lemma:F_star_lemma} we have that each $H \in \calH(n, \Delta, d)$ belongs to some $\calD_d(F^*, z^*, K)$, which implies $G$ is $\calH(n, \Delta, d)$-universal.
\end{proof}

It remain to prove Lemma \ref{lemma:univ_deg_F}.

\begin{proof}[Proof of Lemma \ref{lemma:univ_deg_F}] Let $q\in (0,1)$ be such that $(1-q)^3=1-p$ and observe that $q\approx p/3$. We generate $G = \Gnp$  as the union $G=G_1\cup G_2\cup G_3$ where $G_i = G_{n,q}$ for each $i \in \{1,2,3\}$. This is usually referred to as the \emph{multiple exposure} trick or \emph{sprinkling} (e.g.\ see \cite{JLR} for details). Let $F = F^* \setminus z^*$ and $\Gamma = N_{F^*}(z^*)$. From (D1) and (D2) we conclude
$$
	|\Gamma| = u \le 2d.
$$
Moreover, using the assumption that $F$ is $d$-degenerate and of size at most $3d^2$, one can easily verify that $m_1(F) \le d$ (to prove it, one should use the simple fact that a $d$-degenerate graph on $x$ vertices has at most $d(x-d)+\binom{d}{2}$ edges). Therefore, from Lemma \ref{lemma:factor} we have that for our choice of $K$, a graph $G_1 = G_{n,q}$ w.h.p contains a family of
$$
	t = n /K < n / 4 v(F)
$$
vertex-disjoint copies of $F$, each of which is equipped with an embedding $\tau_i \colon F \hookrightarrow G_1$ (where $i \in [t]$). Let $\calL = \{ \tau_i(\Gamma) \}_{i \in [t]}$ be the set of all images of $\Gamma$ into $G_1$ and consider an arbitrary subset $X \subseteq [n] \setminus \bigcup_{i \in [t]} \tau_i(F)$ of size $3t/4$. Next, a graph $G_2 = G_{n,q}$ w.h.p satisfies the property of Lemma \ref{lemma:S_almost_spanning} for $W = [n] \setminus (X \cup \bigcup_{i \in [t]} \tau_i(F))$ and sufficiently small $\alpha, \gamma > 0$ (we will clarify this later in the description of Phase 2). Moreover, from Lemma \ref{lemma:expansion} we have that $G_3 = G_{n,q}$ w.h.p has the following property for every $\calL' \subseteq \calL$:
$$	
	\big| N_{\calB_1}(\calL') \big| \ge \begin{cases}
		 (1 - \lambda)|\calL'||X| p^{u}, &\text{ if } |\calL'| \le \lambda / p^u, \\
		 (1 - \lambda)|X|, &\text { if } |\calL'| \ge \log n / p^u,
	\end{cases}
$$	
where $\calB_1 = \calB_{G_3}(\calL, X)$. In addition, by Lemma \ref{lemma:expansion_reverse} we obtain that $G_3$ w.h.p has the following property for every $U' \subseteq [n] \setminus \calL$:
$$	
	\big| N_{\calB_2}(U') \big| \ge \begin{cases}
		 (1 - \lambda)|U'||\calL| p^{u}, &\text{ if } |U'| \le \lambda / p^u, \\
		 (1 - \lambda)|\calL|, &\text { if } |U'| \ge \log n / p^u,
	\end{cases}
$$
where $\calB_2 = \calB_{G_3}(\calL, [n] \setminus \calL)$. We show that $G_1 \cup G_2 \cup G_3$ is universal for the family of graphs $\calH(n, \Delta, d) \cap \calD(F^*, z^*)$.

\medskip
\noindent
\textbf{Embed $H$.} Consider any graph $H \in \calH(n, \Delta, d) \cap \calD(F^*, z^*,K)$ and let $D \subseteq V(H)$ and $\{S_w\}_{w \in D}$ be subsets as ensured by Definition \ref{def:degenerate}, with $|D|=t$. We wish to describe an embedding $\phi \colon H \hookrightarrow G$ using only the pseudorandom properties of $G$ mentioned above, which clearly implies the universality. We construct the desired embedding in three phases: In Phase 1 we embed $H[S]$ (where $S := \bigcup_{w \in D} \left(S_w \setminus \{w\}\right)$). Then, in Phase 2 we extend it into and embedding of $H' := H[S \cup R]$ (where $R := V(H) \setminus (S \cup D)$). Finally, in Phase 3 we embed the remaining vertices $D$. The formal details are given bellow:

{\bf Phase 1.} Let $\{\tau_w\}_{w \in D}$ be an arbitrary labelling of $\{\tau_i \colon F \hookrightarrow G_1\}_{i \in [t]}$. Define $\phi_1 \colon H[S] \hookrightarrow G_1$ as the union of all $\tau_w$'s, that is, for each $w \in D$ and $v \in S_w \setminus \{w\}$ set
$$
	\phi_1(v) := \tau_w(f_w(v)).
$$
Recall that $f_w \colon H[S_w] \hookrightarrow F^*$ is an isomorphism which maps $w$ to $z^*$, given by property (D2). Note that this indeed defines an embedding of $H[S]$ as it follows from (D2) and (D3) that each $H[S_w \setminus \{w\}]$ is isomorphic to $F$. Moreover, from the choice of $W$ we have $W \cap (X \cup \phi_1(S)) = \emptyset$.

{\bf Phase 2.} Note that by the property (D4) we have that all vertices $h \in R$ have at most $d$ neighbours in $S$. As $H[R]$ is $d$-degenerate itself, there exists an ordering $(h_1, \ldots, h_m)$ of $R$ such that
$$
	|N_H(h_i, S \cup \{h_1, \ldots, h_{i-1}\})| \le 2d \qquad \text{ for every } 1 \le i \le m.
$$
In particular, from the assumption that $G_2$ satisfies the property of Lemma \ref{lemma:S_almost_spanning} for $W$ and $W \cap (\phi_1(S) \cup X) = \emptyset$ (that is, so far we did not use vertices from $W$), there exists an $S$-embedding $\phi_2 \colon H[S \cup R] \hookrightarrow_S G_2 \setminus X$ which extends $\phi_1$. In particular, this implies $\phi_2 \colon H[S \cup R] \hookrightarrow (G_1 \cup G_2) \setminus X$.

\medskip
\noindent
\emph{Remark.} Note that we did not explicitly specified the values of $\alpha$ and $\gamma$ for which we apply Lemma \ref{lemma:S_almost_spanning}. The value of $\alpha$ is implicitly given by the size of $W$ (which is clearly linear in $n$), whereas the value of $\gamma$ can be derived from the fact that $H[S \cup R]$ is significantly smallest than $n - |X| = n - 3t/4$ --- in particular, we have $H' \in \mathcal{H}(n - |X| - t/4, \Delta)$ and $t$ is linear in $n$.
\medskip

{\bf Phase 3.} Finally, let $U := V(G) \setminus \phi_2(S \cup R)$ and observe that $|U| = |D|$ and $X \subseteq U$. Let $\calB = \calB_{G_3}(\{\tau_w(\Gamma)\}_{w \in D}, U)$ and note that $\calB_1 \subseteq \calB$. In order to finish off the embedding it suffices to show that $\calB$ contains a perfect matching. Indeed, assuming this is true let $\xi \colon D \rightarrow U$ be a bijection such that for each $w \in D$ we have that $\tau_w(\Gamma)$ is connected to $\xi(w)$ in $\calB$. That is,
$$
	\phi_2(N_H(w)) = \tau_w(f_w(N_H(w))) \subseteq N_{G_3}(\xi(w))
$$
for all $w\in D$. Therefore, by defining $\phi$ to be an extension of $\phi_2$ to $D$ by setting $\phi(w) = \xi(w)$ for every $w \in D$ we get an embedding of $H$ in $G$.

In order to show that $\calB$ contains a perfect matching we verify Hall's condition (Theorem \ref{thm:hall}). First, consider a family $\calL' \subseteq \calL$ of size at most $\log n / p^u$. If $\calL'$ is larger than $\lambda / p^u$ then let $\calL'' \subseteq \calL'$ be an arbitrary family of size $\lambda / p^{u}$, and otherwise let $\calL'' = \calL$. Then
$$
	|N_\calB(\calL')| \ge |N_{\calB_1}(\calL'')| \ge |\calL''||X|p^u.
$$
As $|X|p^u > \log^2 n$ and $|\calL''| \ge \lambda |\calL'| / \log n$, we obtain the desired inequality $|N_\calB(\calL')| \ge |\calL'|$. Second, for a family $\calL' \subseteq \calL$ of size $\log n / p^u \le |\calL'| \le t / 2$ we have
$$
	|N_{\calB}(\calL')| \ge |N_{\calB_1}(\calL')| \ge (1 - \lambda)|X| \ge (1 - \lambda)3t / 4 > t / 2,
$$
as required. The same calculation shows $|N_\calB(U')| \ge |U'|$ for every $U' \subseteq U$ of size $|U'| \le t / 2$. Therefore, by Theorem \ref{thm:hall} we conclude that $\calB$ contains a perfect matching.
\end{proof}

%% file: bounded_degree.tex
\section{Universality for bounded-degree graphs}
\label{sec:main}

In this section we prove Theorem \ref{thm:spanning} which is our main result. The basic proof strategy is quite similar to the proof of Theorem \ref{thm:spanning_deg}, however the details are somewhat more complicated. The main obstacle lies in the finishing part where we need to find a perfect matching in an auxiliary bipartite graph between sets of size $\Delta$ and a set of unused vertices.  If $p= \tilde \Omega(n^{-1/\Delta})$ (which is the edge-probability we aim to beat), then it could be done in the same way as in the proof of Theorem \ref{thm:spanning_deg}. Indeed, we would then have that the expected number of common neighbours of every $\Delta$-subset is large. Even though for $p = n^{-\eps - 1/\Delta}$, for some small $\eps > 0$, we do not have the property that every $\Delta$-subset has a non-empty common neighbourhood there are certainly $\Delta$-subsets for which this is true (and in fact, there are quite a lot of them). Thus, instead of `blindly' embedding $H[S_w \setminus \{w\}]$'s (see the proof strategy described in Section \ref{sec:degenerate}) we do it in such a way that each $N_H(w) \subseteq S_w$ is being embedded into  a `good' $\Delta$-subset. Moreover, we make sure that the common neighbours of different $N_H(w)$'s are intertwined in a way that gives us similar freedom for embedding the remaining part of $H$ as we had in the proof of Theorem \ref{thm:spanning_deg}. This is accomplished by constructing an \emph{absorbing} structure for $D$ and showing that such a structure appears in $\Gnp$ (Lemma \ref{lemma:delta_absorber}). We now give a brief outline of the strategy used to embed one graph $H$ and make the previous discussion more precise.

\begin{enumerate}
	\item Similarly to the proof of Theorem \ref{thm:spanning_deg}, we first choose a subset of vertices $D \subseteq V(H)$ such that every two vertices in it are far apart (thus $D$ is an independent set). Next, for each $w \in D$ choose a small subset $S_w \subset V(H)$ such that $S_w$ contains $w$ and its neighbourhood and no vertex outside of $S_w$ sends more than $\Delta-1$ edges into $S_w$. Importantly, $S_w$'s are chosen such that each induced graph $H[S_w]$ is isomorphic to the same (connected) graph $F^*$ and each $w$ has the same role in $F^*$ (say, it has the role of a vertex $z^* \in F^*$). By virtue of being connected, we have that each vertex in $S_w$ is `close' to $w$ (as $S_w$ is small) and, since every two vertices in $D$ are far apart, $S_w$'s are pairwise disjoint and there are no edges between them.

	\item At this point we slightly diverge from the proof of Theorem \ref{thm:spanning_deg}. Recall that in the former case it was enough to find any $F$-matching in $G$ of size $|D|$, where $F := F^* \setminus z^*$, and use it to define an embedding of $H[S_w \setminus \{w\}]$'s. As mentioned before, such an approach fails here: a typical $\Delta$-subset has $n p^\Delta \to 0$ common neighbours thus if we just take an arbitrary $F$-matching then the image of a typical $N_H(w)$ will not have any common neighbours and, consequently, $G$ will contain no candidates for $w$.

	The main new ingredient is the following lemma which captures the property used to finish off the embedding in the proof of Theorem \ref{thm:spanning_deg}. We need the following definition:

	\begin{definition}[$Y$-robustness] \label{def:robust}
	Given a bipartite graph $B$ on vertex classes $Z$ and $X \cup Y$ with $|X| < |Z| < |X| + |Y|$, we say that $B$ is \emph{$Y$-robust} if for every subset $Y' \subseteq Y$ of size $|Y'| = |Z| - |X|$, $B$ contains a perfect matching between $Z$ and $X \cup Y'$.
	\end{definition}

	\begin{lemma} \label{lemma:delta_absorber}	
	There exist positive constants $\beta_1, \beta_2$ with $\beta_1 + \beta_2 < 1$ such that the following holds. Let $\Delta\geq 3$ be an integer, let $F$ be a graph with maximum degree $\Delta$ and $\Gamma \subseteq V(F)$ a subset of size $|\Gamma| \le \Delta$ such that $\deg_F(v) \le \Delta - 1$ for every $v \in \Gamma$.  Given a positive constant $\nu < (100v(F))^{-1}$,  if
	$$
		p \ge (n^{-1}\log^3 n)^{1/(\Delta - 1/2)}
	$$
	then $G = \Gnp$ w.h.p contains:
	\begin{itemize}
		\item a collection $\{\tau_i \colon F \hookrightarrow G\}_{i \in [t]}$ of vertex-disjoint embeddings of $F$, where $t = \nu n$, and
		\item disjoint subsets $X, Y \subseteq V(G) \setminus \bigcup_{i \in [t]} \tau_i(F)$ of size $|X| = (1 - \beta_1)t$ and $|Y| = (\beta_1 + \beta_2) t$,
	\end{itemize}
	such that the bipartite graph $\calB_G(\{\tau_i(\Gamma)\}_{i \in [t]}, X \cup Y)$ is $Y$-robust.
	\end{lemma}
	
	The proof of Lemma \ref{lemma:delta_absorber} is quite involved thus we postpone it until the end of this section.

	Having Lemma \ref{lemma:delta_absorber} at hand, we apply it with $F = F^* \setminus z^*$ and $\Gamma = N_F(z^*)$ to obtain an $F$-matching. We then associate such copies of $F$ with $H[S_w \setminus \{w\}]$'s which defines an embedding of $H[S]$, where $S := \bigcup_{w \in D} S_w \setminus \{w\}$ (see Figure \ref{fig:phase_bound1}).

	\begin{figure}[h!]
		\centering
		\subfloat[First phase (step 2.)]{\raisebox{0.48cm}{\includegraphics[scale=0.9]{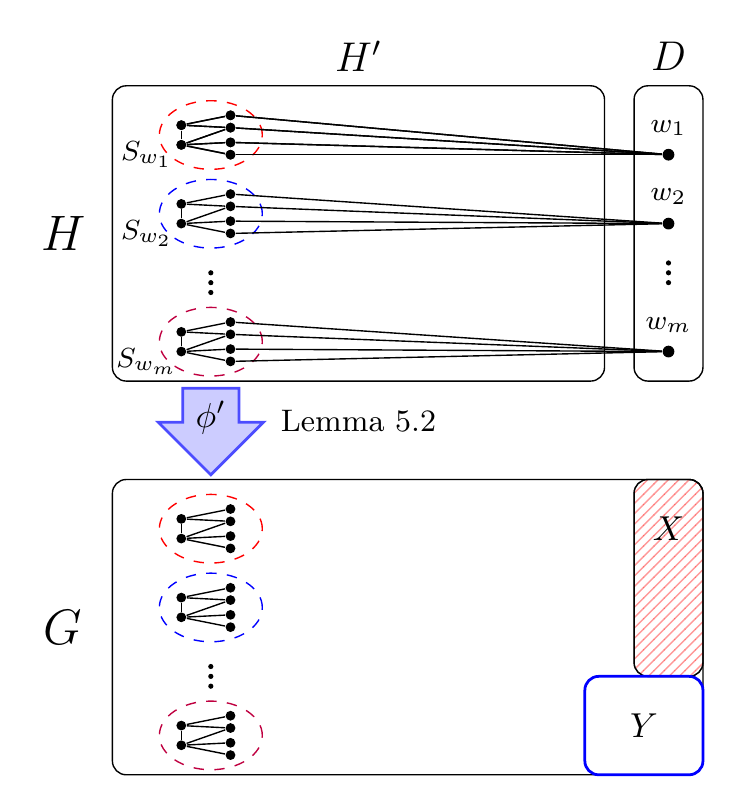}} \label{fig:phase_bound1}}
		\hspace{1.5cm}
		\subfloat[Second phase (step 3.)]{\includegraphics[scale=0.9]{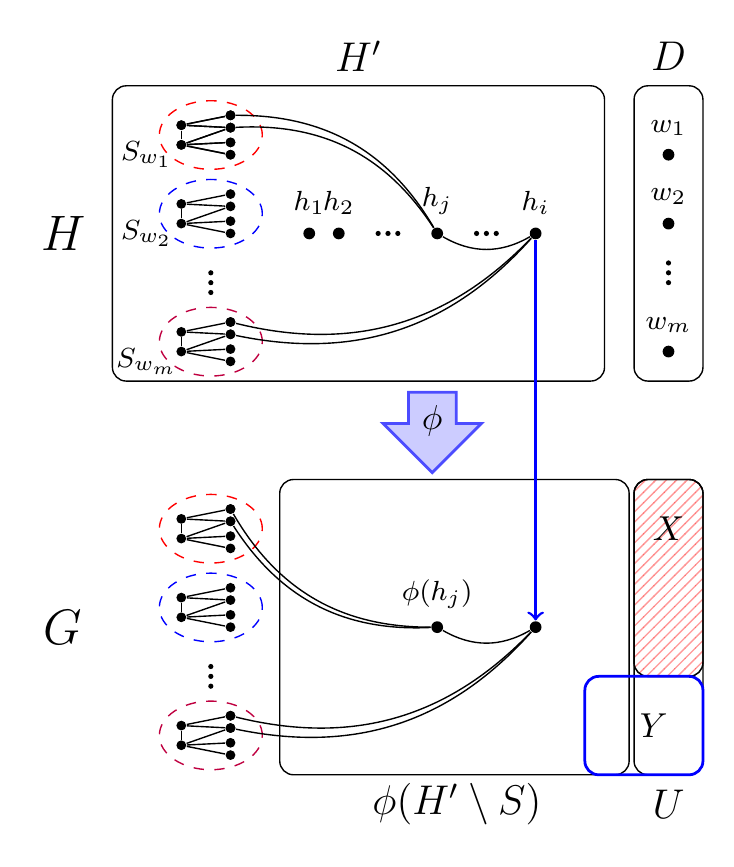} \label{fig:phase_bound2}}
		\caption{Embedding a graph $H$ with maximum degree $3$.}
	\end{figure}

	\item Next, we extend an embedding of $H[S]$ into an embedding $\phi$ of $H' := H \setminus D$. Unlike in the proof of Theorem \ref{thm:spanning_deg} where it was enough to embed it in such a way that $X$ is being avoided, here we have a more difficult task: in order to make use of Lemma \ref{lemma:delta_absorber}, it is not enough to avoid $X$ but we also need to use \emph{all} the vertices which are not in $X \cup Y$ (see Figure \ref{fig:phase_bound2}). This is done by further splitting $H'$ into two parts and applying Lemma \ref{lemma:delta_S_embedding} to embed each of them. For now we omit the details on how exactly this can be done.

	\item Finally, assuming we can extend an embedding of $H[S]$ into an embedding of $H'$ such that the set $U \subseteq V(G)$ of unused vertices satisfies $X \subseteq U \subseteq X \cup Y$, we finish the embedding by finding a perfect matching in the auxiliary graph $\calB_G(\calL, U)$, where $\calL = \{\phi(N_H(w)) \colon w \in D\}$. The existence of such a perfect matching follows immediately from $Y$-robustness property of $\calB_G(\calL, X \cup Y)$.
\end{enumerate}

In the rest of the section we make this strategy precise. We start by introducing few lemmata and definitions akin to those in Section \ref{sec:degenerate}.



The following is a version of Definition \ref{def:degenerate} tailored to the proof of Theorem \ref{thm:spanning}. There are two differences: (i) we do not explicitly specify a bound on the degree of $w$ in (D1) as, in general, we will not be able to use any bound  better  than $\Delta$; and (ii) the bound in (D4) is slightly stronger.

\begin{definition} \label{def:bounded_deg}
	Let $\Delta, K\in \mathbb{N}$, let $F^*$ be a graph with $\Delta(F^*) \leq \Delta$ and let $z^* \in V(F^*)$. Define $\calD_\Delta(F^*, z^*, K)$ to be the family of all graphs $H$ for which the following holds: there exist a set $D \subseteq V(H)$ of size at least $|V(H)| /K$ and a family of subsets $\{S_w\}_{w \in D}$ (where each $S_w\subseteq V(H)$) such that the following is true:
	\begin{enumerate}[(D1)]
		\item for each $w \in D$ we have $\{w\} \cup N_H(w) \subseteq S_w$ and
		\item there exists an isomorphism $f_w \colon H[S_w] \hookrightarrow F^*$ which maps $w$ to $z^*$;
		\item $S_w \cap S_{w'} = \emptyset$ and there are no edges between $S_w$ and $S_{w'}$, for every $w \neq w' \in D$;
		\item for each vertex $w \in V(H) \setminus \bigcup_{w \in D} S_w$ we have $|N_H(w) \cap \bigcup_{w \in D} S_w| \le \Delta - 1$.
	\end{enumerate}
\end{definition}

Again, as in Section \ref{sec:degenerate}, we show that each $\calH(n, \Delta)$ can be covered with $\calD_\Delta(F^*, z^*,K)$'s where $F^*$ is a small graph.

\begin{lemma} \label{lemma:F_star_bounded}
	Let $\Delta \in \mathbb{N}$. Then there exists $K = K(\Delta)$ such that for every graph $H$ with maximum degree $\Delta$ there exists a graph $F^*$ with at most $2\Delta$ vertices and a vertex $z^* \in V(F^*)$ such that $H \in \calD_\Delta(F^*, z^*, K)$.
\end{lemma}
\begin{proof}[Proof (Sketch).] The proof is more or less the same as the proof of Lemma \ref{lemma:F_star_lemma} so we only give a brief sketch and  emphasize the main differences.

Let $H$ be a graph with $n$ vertices and maximum degree at most $\Delta$. Let $D'\subseteq V(H)$ be a largest subset of vertices such that every two are at distance at least $k$ (where $k$ is chosen to be large enough). Lemma \ref{lemma:distance} guarantees that $D'$ is of size at least
$$
	|D'| \ge n / \Delta^{k+1}.
$$
Next, for each vertex $w\in D'$ we define a graph $S_w$ as follows:
Start with $S_w:=\{w\}\cup N_H(w)$ and as long as there exists a vertex $v\in V(H)$ with $|N_H(v) \cap S_w |=\Delta$, pick such a vertex and update $S_w:=S_w\cup\{v\}$. Note that after $t$ steps the set $S_w$ is of size $t+ \deg_H(v)+1$ and $H[S_w]$ contains at least $\deg_H(v) + \Delta t$ edges.
Moreover, as $\Delta(H)\leq \Delta$ we clearly have
$$
	e(H[S_w])\leq \Delta(t+\deg_H(v)+1)/2.
$$
Combining these two estimates on $e(H[S_w])$ we obtain
$$
	\deg_H(v)+\Delta t\leq \Delta(t+\deg_H(v)+1)/2
$$
which is equivalent to
$$
	t\leq \frac{\deg_H(v)(\Delta-2)+\Delta}{\Delta}.
$$
Using the fact that $\deg_H(v)\leq \Delta$, the above inequality  can only hold if $t\leq \Delta-1$. All in all, each $S_w$ is of size at most $2\Delta$.

From now on the proof goes exactly as the proof of Lemma \ref{lemma:F_star_lemma} so we omit the details.
\end{proof}

Having the previous lemma in mind, Theorem \ref{thm:spanning} reduces to the following lemma.

\begin{lemma} \label{lemma:univ_F}
	Let $\Delta\geq 3$ be an integer and let $K\in \mathbb{N}$. Let $F^*$ be a graph with at most $2\Delta$ vertices and maximum degree at most $\Delta$, and let $z^* \in V(F^*)$ be a vertex in $F^*$. If
	$$
		p \ge (n^{-1}\log^3 n)^{1/(\Delta - 1/2)}
	$$
	then $\Gnp$ is w.h.p universal for the family of graphs $\calH(n, \Delta) \cap \calD_\Delta(F^*, z^*, K)$.
\end{lemma}

\begin{proof}[Proof of Theorem \ref{thm:spanning}]
	Let $K$ be a constant given by Lemma \ref{lemma:F_star_bounded}. It follows from Lemma \ref{lemma:univ_F} and the union bound that $G = \Gnp$ is w.h.p universal for the family of graphs
	$$
		\calH(n, \Delta) \cap \left( \bigcup_{\substack{F^* \\ z^* \in F^*}} \calD_\Delta(F^*, z^*, K) \right),
	$$
	where the union bound goes over all pairs $(F^*, z^*)$ where $F^*$ is a graph with at most $2\Delta$ vertices and $z^* \in V(F^*)$. From Lemma \ref{lemma:F_star_bounded} we have that each $H \in \calH(n, \Delta)$ belongs to some $\calD_\Delta(F^*, z^*, K)$, which implies $G$ is $\calH(n, \Delta)$-universal.	
\end{proof}

It remains to prove Lemma \ref{lemma:univ_F}. The proof relies on Lemma \ref{lemma:delta_absorber} which is then proved in the next section.

\begin{proof}[Proof of Lemma \ref{lemma:univ_F}]
Let $F = F^* \setminus z^*$, $\Gamma = N_{F^*}(z^*)$ and $\nu= \min\{1/2K, (200 v(F))^{-1}\}$. Let $\beta_1,\beta_2$ be the constants given by Lemma \ref{lemma:delta_absorber}. Then the graph $G_1 = \Gnp$ w.h.p contains a collection $\{\tau_i \colon F \hookrightarrow G_1\}_{i \in [t]}$ of vertex-disjoint embeddings of $F$ and disjoint subsets $X, Y \subseteq V(G) \setminus \bigcup_{i \in [t]} \tau_i(F)$ such that $\calB = \calB_{G_1}(\{\tau_i(\Gamma)\}_{i \in [t]}, X \cup Y)$ is $Y$-robust.

Next, let $W_1 \cup W_2 \cup V' = Y$ be an arbitrary partition such that $|W_1| = |W_2| = \alpha n$ for some $\alpha \ll \beta_1 + \beta_2$. By a slight abuse of notation, we write $|W_i| = o(t)$ to denote that the size of $W_i$ can be arbitrarily small with respect to $t$ (as we can choose $\alpha$ to be arbitrarily small). The graph $G_2 = \Gnp$ w.h.p satisfies the property of Lemma \ref{lemma:delta_S_embedding} with $W_i$ (as $W$) for $i \in \{1,2\}$. We show that $G = G_1 \cup G_2$ is universal for the family of graphs $\calH(n, \Delta) \cap \calD_\Delta(F^*, z^*, K)$.

\medskip
\noindent
\textbf{Embed $H$.} Consider a graph $H \in \calH(n, \Delta) \cap \calD(F^*, z^*, K)$ and let $D$ and $\{S_w\}_{w \in D}$ be subsets as given by Definition \ref{def:degenerate}. Moreover, remove arbitrary elements from $D$ in order to make $|D| = t$. Set $S := \bigcup_{w \in D} (S_w \setminus \{w\})$ and $R := V(H) \setminus \bigcup_{w \in D} S_w$. Similarly as in the proof of Theorem \ref{thm:spanning_deg}, we proceed in three phases: In Phase 1 we embed $H[S]$ using the obtained family $\{\tau_i \colon F \hookrightarrow G_1\}_{i \in [t]}$ of vertex-disjoint copies of $F$. In Phase 2 we extend this embedding into an embedding of $H' := H[S \cup R]$ with the following requirement: we want that all the vertices in $V(G) \setminus (X \cup Y)$ and non of the vertices from $X$ to be used. This allows us to use the $Y$-robustness of $\calB$ to finish off the embedding in Phase 3 by finding a perfect matching in an auxiliary graph as before.

Let us emphasize once again the additional difficulty arising in Phase 2 of our embedding scheme: In the proof of Theorem \ref{thm:spanning_deg} we could finish the embedding of $H$ as long as no vertex from $X$ was used. Here things do not work that smooth and we have to embed $H[S \cup R]$ in such a way that not only the vertices in $X$ are not used but we also need that \emph{all} the vertices in $V(G) \setminus (X \cup Y)$ are used.

Bellow is a formal description of each of the phases.

\textbf{Phase 1.} This phase is identical to Phase 1 in the proof of Theorem \ref{thm:spanning_deg}. Let us relabel $\{\tau_i \colon F \hookrightarrow G_1\}_{i \in [t]}$ as $\{\tau_w\}_{w \in D}$ in an arbitrary way and define $\phi_1 \colon H[S] \hookrightarrow G_1$ as the union of all $\tau_w$'s. That is, for each $w \in D$ and $v \in S_w \setminus \{w\}$ we set
$$
	\phi_1(v) := \tau_w(f_w(v)).
$$
Recall that $f_w \colon H[S_w] \hookrightarrow F^*$ is an isomorphism which maps $w$ to $z^*$, given by property (D2). Note that this defines an embedding of $H[S]$. Indeed, from (D2) and (D3) we have that each $H[S_w \setminus w]$ is an isomorphic copy of $F$, they do not overlap and there are no edges between them.

\textbf{Phase 2.} In this phase our aim is to extend $\phi_1$ into $\phi_2 \colon H[S \cup R] \hookrightarrow G \setminus X$ such that $V(G) \setminus (X \cup Y) \subseteq \phi_2(S \cup R)$. In order to do so, we first partition $R = R_1 \cup I \cup R_2$ according to the following easy claim which will be proven later.

\begin{claim} \label{claim:partition_R}
There exists a partition $R = R_1 \cup I \cup R_2$ such that the following holds:
\begin{enumerate}[(a)]		
	\item $|R_2| = \beta_2 t / 2$ and $|I| \ge |R_2| / \Delta^2$,
	\item $I$ is a set of isolated vertices in $H[S \cup R_1 \cup I]$, and
	\item $|N_H(h) \cap (S \cup R_1 \cup I)| \le \Delta - 1$ for every $h \in R_2$.
\end{enumerate}
\end{claim}

Next, we proceed in three sub-phases: in Phase 2.a we extend $\phi_1$ to an embedding of $H[S \cup R_1]$ such that all but at most $|I|$ vertices from $V(G) \setminus (X \cup Y)$ are used. This is done using Lemma \ref{lemma:delta_S_embedding}. Next, using the fact that $I$ is a set of \emph{isolated} vertices in $H[S \cup R_1 \cup I]$ (Property (b)) we can map $I$ onto the remaining vertices from $V(G) \setminus (X \cup Y)$ (and some vertices from $Y$) in an arbitrary way. Finally, Property (c) enables us to use Lemma \ref{lemma:delta_S_embedding} to extend this to an embedding of $H[S \cup R]$ with the desired properties. We now give the formal details.

\textbf{Phase 2.a.} Choose a subset $V'' \subseteq V'$ such that $X' := X \cup W_2 \cup V''$ is of size $|X'| = |D| + |R_2|$. This is indeed possible as
$$
	|X| + |W_2| = (1 - \beta_1)t + o(t) < t + \beta_2 t / 2 = |D| + |R_2|
$$
and $|V'| = (\beta_1 + \beta_2) t - o(t)$. Such a choice of $X'$ implies
\begin{equation} \label{eq:v2}
	|S| + |R_1| = n - |D| - |R_2| - |I| = n - |X'| - |I| \le n - |X'| - \Theta(n).
\end{equation}
In addition, from Property (D4) we have
$$
	|N_H(h, S)| \le \Delta - 1 \quad \text{ for every } h \in R_1,
$$
and therefore one can apply Lemma \ref{lemma:delta_S_embedding} with $W_1$ (as $W$), $X'$ (as $X$) and $\phi_1$ (as $\phi'$) to obtain an $S$-embedding $\phi_2' \colon H[S \cup R_1] \hookrightarrow_S G_2 \setminus X'$ which extends $\phi_1$. Note that then $\phi_2'$ is an embedding of $H[S \cup R_1]$ into $G \setminus X'$.

\textbf{Phase 2.b.}
Let $U \subseteq V(G) \setminus X'$ denote the set of `unused' vertices, that is,
$$
	U = V(G) \setminus (X' \cup \phi_2'(S \cup R_1)).
$$
From $|S| + |R_1| = n - |X'| - |I|$, as derived in \eqref{eq:v2}, we conclude $|U| = |I|$. Therefore, taking an arbitrary bijection between $I$ and $U$ we extend $\phi_2'$ to an embedding $\phi_2'' \colon H[S \cup R_1 \cup I] \hookrightarrow G \setminus (X \cup Y_2)$. Importantly, note that at this point we have $V(G) \setminus (X \cup Y) \subseteq \phi_2''(S \cup R_1 \cup I)$.

\textbf{Phase 2.c.}
Note that
$$
	|S| + |R| = n - |D| = n - t = n - |X| - \beta_1 t = n - |X| - \Theta(n)
$$
and $\phi_2''(S') \cap W_2 = \emptyset$, where $S' = S \cup R_1 \cup I$. Therefore, by using Property (c) one can apply Lemma \ref{lemma:delta_S_embedding} with $W_2$ (as $W$), $X$ and $\phi_2''$ (as $\phi'$) to obtain an $S'$-embedding $\phi_2 \colon H[S' \cup R_2] \hookrightarrow_{S'} G_2 \setminus X$. Note that as $\phi''_2$ embeds $H[S']$ into $G \setminus (X \cup Y_2)$, we conclude that $\phi_2$ is an embedding of $H[S \cup D]$ into $G \setminus X$ such that $V(G) \setminus (X \cup Y) \subseteq \phi_2(S \cup D)$, as desired.

\textbf{Phase 3.} In the last phase we embed all the vertices from $D$. To this end, let $Y' \subseteq Y \setminus \phi_2(S \cup R)$ denote the set of unused vertices from $Y$ and note that these are the only unused vertices from $V(G) \setminus X$. From the assumption that $\calB = \calB_{G_1}(\{\tau_w(\Gamma)\}_{w \in D}, X \cup Y)$ is $Y$-robust we conclude the existence of a perfect matching between $\{\tau_w(\Gamma)\}_{w \in D}$ and $X \cup Y'$. That is, there exists a bijection $\xi \colon D \rightarrow X \cup Y'$ such that
$$
	\phi_2(N_H(w))  = \tau_w(f_w(N_H(w))) = \tau_w(\Gamma)  \subseteq N_{G_1}(\xi(w)).
$$
All in all, the extension $\phi$ of $\phi_2$ to $D$ given by $\phi(h) = \xi(h)$ defines an embedding of $H$ into $G$. This completes the embedding.

It remains to prove Claim \ref{claim:partition_R}.

\begin{proof}[Proof of Claim \ref{claim:partition_R}]
First, observe that there exists $k \in \{1, \ldots, \Delta^2\}$ and an independent set $I \subseteq V(R)$ in $H$ such that
$$
	|I| = \beta_2 t / 2k
$$	
and the following holds for every $h \in I$:
\begin{enumerate}[(i)]	
	\item the distance between $h$ and any vertex $h' \in (I \setminus \{h\}) \cup S$ is at least $5$, 	
	\item $|(N_{H}(h) \cup N_H^2(h)) \setminus \{h\}| = k$ for every $h \in I'$.	
\end{enumerate}
Such set can be obtained as follows: There are at least $|R| - \Delta^5|S| > n/2$ vertices $h \in R$ which are of distance at least $5$ from any vertex in $S$. Let $R' \subseteq R$ denote the set of such vertices. By Lemma \ref{lemma:distance} there exists a subset $I' \subseteq R'$ of size at least $|R'| / \Delta^5$ such that every two vertices from $I'$ are of distance at least $5$. Next, by the pigeonhole principle there exists $k \in \{1, \ldots, \Delta^2\}$ such that the subset $I'_k \subseteq I$ of all vertices $h \in I'$ with $|(N_H(h) \cup N_{H}^2(h)) \setminus \{h\}| = k$ is of size at least $|I'| / \Delta^2 \ge \beta t / 4k$. Therefore, we can choose $I \subseteq I'_k$ to be an arbitrary subset of the desired size.

Set $R_2 := \left(\bigcup_{h \in I} N_H(h) \cup N_H^2(h) \right) \setminus I$ and $R_1 := R \setminus (R_2 \cup I)$. From (i), (ii) and the size of $I$ we conclude
$$
	|R_2| = k|I| = \beta_2 t / 2.
$$
Furthermore, every vertex in $(N_H(h) \cup N_H^2(h)) \setminus \{h\}$ either has a neighbour in $(N_H(h) \cup N_H^2(h)) \setminus \{h\}$ or its only neighbour in $H$ is $h$. In any case, we conclude
$$
	|N_H(h) \cap (S \cup R_1 \cup I)| \le \Delta - 1,
$$
as desired.
\end{proof}
This completes the proof of Lemma \ref{lemma:univ_F}.
\end{proof}

\subsection{Proof of Lemma \ref{lemma:delta_absorber}}
\label{sec:delta_absorber}

In this section we prove Lemma \ref{lemma:delta_absorber}. 
Our starting point is the following simple lemma from \cite{montgomery2014embedding}.

\begin{lemma}[{\cite[Lemma 2.8]{montgomery2014embedding}}] \label{lemma:richard}
	There exists $m_0 \in \mathbb{N}$ such that for every $m \ge m_0$ there exists a bipartite graph $B$ on vertex classes $Z$ and $X \cup Y$ with the following properties:
	\begin{enumerate}[(i)]
		\item $\Delta(B) \le 40$ and every vertex in $Z$ has degree exactly $40$,
		\item $|Z| = 3m$ and $|X| = |Y| = 2m$,
		\item $B$ is $Y$-robust (see Definition \ref{def:robust}).
	\end{enumerate}
\end{lemma}

We construct an absorbing structure as follows: Let $B$ be a bipartite graph on vertex classes $Z$ and $X \cup Y$ given by Lemma \ref{lemma:richard} for suitably chosen $m$. We obtain the graph $A_B$ from $B$ (an \emph{absorber} based on $B$) by replacing each vertex $z \in Z$ with a distinct copy of $F$, denoted by $F_z$, and for each edge $zv \in B$ connect $v$ to all the vertices in $\Gamma_z$ (here $\Gamma_z \subseteq V(F_z)$ denotes the subset which corresponds to $\Gamma \subseteq V(F)$). Then the auxiliary bipartite graph $\calB = \calB_{A_B}(\{\Gamma_z\}_{z \in Z}, X \cup Y)$ has the property that $\Gamma_z v \in \calB$ if and only if $zv \in B$ thus $\calB$ inherits the $Y$-robustness from $B$. In particular, if $A_B \subseteq G$ for some graph $G$, then such copies of $F$ define a collection of vertex-disjoint embeddings of $F$ which together with $X$ and $Y$ satisfy the property of Lemma \ref{lemma:delta_absorber}. Therefore, we aim to show that typically $A_B \subseteq \Gnp$. Similar ideas were used in \cite{kwan2016almost,montgomery2014embedding}.

For larger $\Delta$ (say, $\Delta\geq 50$) one can show relatively easily that $A_B \subseteq \Gnp$ for $p \ge n^{-\eps - 1/\Delta}$, for some small $\eps > 0$, using the result of Riordan \cite{riordan2000spanning}. However, for small $\Delta$ the bound on $p$ obtained from \cite{riordan2000spanning} does not suffice. The main idea is to alter the graph $B$ from Lemma \ref{lemma:richard} in order to obtain a very sparse and nicely structured graph $A_B$ which can then be embedded with the help of the lemmata from Section \ref{sec:F_matching}.

As the first step we decrease the degree of vertices in $Z$ down to $3$.

\begin{lemma} \label{lemma:max_deg_3}
	There exists $m_0, C \in \mathbb{N}$ such that for every $m \ge m_0$ there exists a bipartite graph $B$ on vertex classes $Z$ and $X \cup Y$ with the following properties:
	\begin{enumerate}[(i)]
		\item $\Delta(B) \le 40$ and every vertex in $Z$ has degree at most $3$,		
		\item $|Z| = Cm$, $|X| = (C-1)m$ and $|Y| = 2m$, and
		\item $B$ is $Y$-robust.
	\end{enumerate}
\end{lemma}
\begin{proof}
	Let $B_0$ be the graph with vertex classes $Z_0$ and $X_0 \cup Y$ given by Lemma \ref{lemma:richard} for sufficiently large $m$. Set $i = 0$ and as long as $Z_i$ contains a vertex of degree at least $4$ apply the following \emph{splitting} operation: pick such a vertex $v$ and consider an arbitrary partition $N_1 \cup N_2 = N_{B_i}(v)$ of the neighbourhood of $v$ such that $|N_1| \ge |N_2| = 2$. We obtain the graph $B_{i+1}$ from $B_i$ by removing the vertex $v$, adding vertices $\{v_1, v_2, u\}$ and placing an edge between $v_j$ and all the vertices in $\{u\} \cup N_j$, for $j \in \{1,2\}$ (see Figure \ref{fig:splitting}). The graph $B_{i+1}$ has vertex classes $Z_{i+1} = (Z_i \setminus \{v\}) \cup \{v_1, v_2\}$ and $X_{i+1} \cup Y = X_i \cup \{u\} \cup Y$. Set $i := i + 1$ and proceed to the next step.

	\begin{figure}[h!]
  		\centering
  		\includegraphics[scale=0.6]{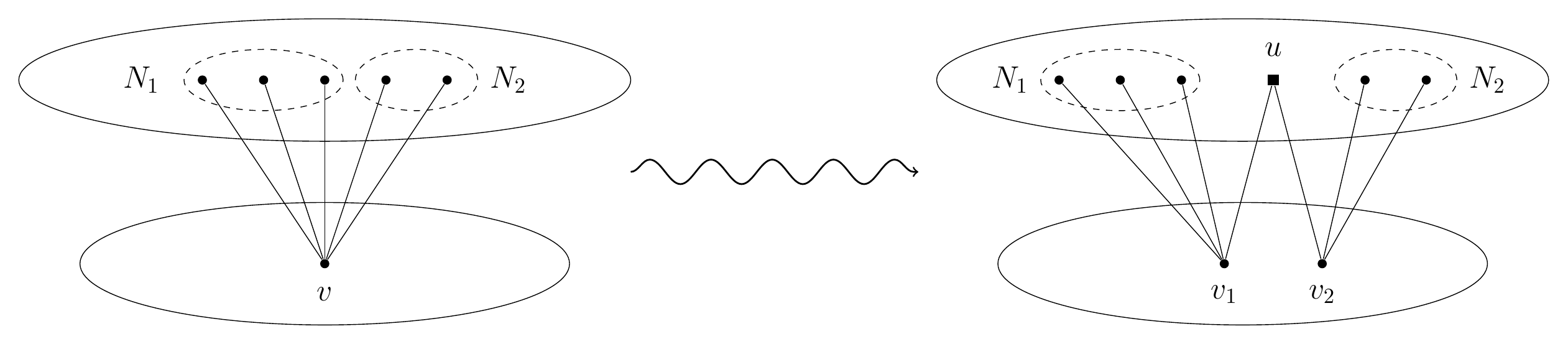}  		
  		\caption{Splitting a vertex $v$.}
  		\label{fig:splitting}
	\end{figure}
	
	Note that each such operation does not change the degree of any vertex in $B_i \setminus \{v\}$. Furthermore, we have $\deg_{B_{i+1}}(v_2) = 3$ and $\deg_{B_{i+1}}(v_1) = \deg_B(v) - 1$ thus the whole process terminates after exactly $38 \cdot 3m$ steps (recall that the initial size of $Z$ is $3m$ and every vertex in $Z$ has degree exactly $40$). In particular, we have $|Z_i| = 114m$, $|X_i| = 113m$ (in each step the size of both $X_i$ and $Z_i$ increases by $1$) and $|Y| = 2m$.

	Next, we claim that if $B_i$ is $Y$-robust then so is $B_{i+1}$. Consider an arbitrary subset $Y' \subseteq Y$ such that $|X_i| + |Y'| = |Z_i|$ and let $\xi \colon Z_i \rightarrow X_i \cup Y'$ denote a perfect matching. Note that then also $|X_{i+1}| + |Y'| = |Z_{i+1}|$ since the size of both $X_{i+1}$ and $Z_{i+1}$ increases by 1 (as already noted before). For each vertex $z \in Z_i \setminus \{v\}$ we match $z$ to $\xi(z)$, which saturates all the vertices in $X_{i+1} \cup Y'$ except $\{u, \xi(v)\}$ and all the vertices in $Z_{i+1}$ except $\{v_1, v_2\}$. If $\xi(v) \in N_1$ then we match $v_1$ to $\xi(v)$ and $v_2$ to $u$, and otherwise match $v_2$ to $\xi(v)$ and $v_1$ to $u$. This gives a perfect matching between $Z_{i+1}$ and $X_{i+1} \cup Y'$, as required.

	To summarise, upon the termination of the procedure we obtain a graph $B$ with vertex classes $Z$ and $X \cup Y$ which is $Y$-robust and $|Z| = 114m$, $|X| = 113m$ and $|Y|=2m$, thus the lemma holds for $C = 114$.
\end{proof}


The next lemma shows that we can further sparsify $B$ such that vertices of degree at least $3$ are far apart. This structural property will play an important role in our embedding scheme.

\begin{lemma} \label{lemma:subdivision}
	Let $B$ be a graph with vertex classes $Z$ and $X \cup Y$ which is $Y$-robust. Consider a graph $B'$ obtained from $B$ by replacing each edge $zv \in B$ by a path $P_{zv}$ of length $L$ (where $L$ is odd), in such a way that all paths are internal disjoint (that is, $B'$ is an $(L-1)$-\emph{subdivision} of $B$). Then the following holds:
	\begin{itemize}
		\item $B'$ is a bipartite graph,
		\item there exists a unique partition $Z' \cup X'$ of internal vertices of such paths such that $Z \cup Z'$ and $X \cup X' \cup Y$ form vertex classes of $B'$, and
		\item $B'$ is $Y$-robust.
	\end{itemize}
\end{lemma}		
\begin{proof}
	First, observe that $|X'| = |Z'|$ as each edge in $B$ contributes the same number of vertices to both sides. Consider a subset $Y' \subseteq Y$ such that $|Z \cup Z'| = |X \cup X'| + |Y'|$. Then $|Z| = |X| + |Y'|$ thus there exists a perfect matching $\xi \colon Z \to X \cup Y'$. We now form a perfect matching between $Z \cup Z'$ and $X \cup X' \cup Y'$ as follows: for each for $z \in Z$ choose a perfect matching in $P_{z \xi(z)}$ (blue matching in Figure \ref{fig:subdivide}); similarly, for each edge $z v \in B$ where $z \in Z$ and $\xi(z) \neq v \in X \cup Y$ choose a matching in $P_{zv}$ which contains all vertices except $\{z,v\}$ (red matching in Figure \ref{fig:subdivide}).

	\begin{figure}[h!]
		\centering
		\includegraphics[scale=0.6]{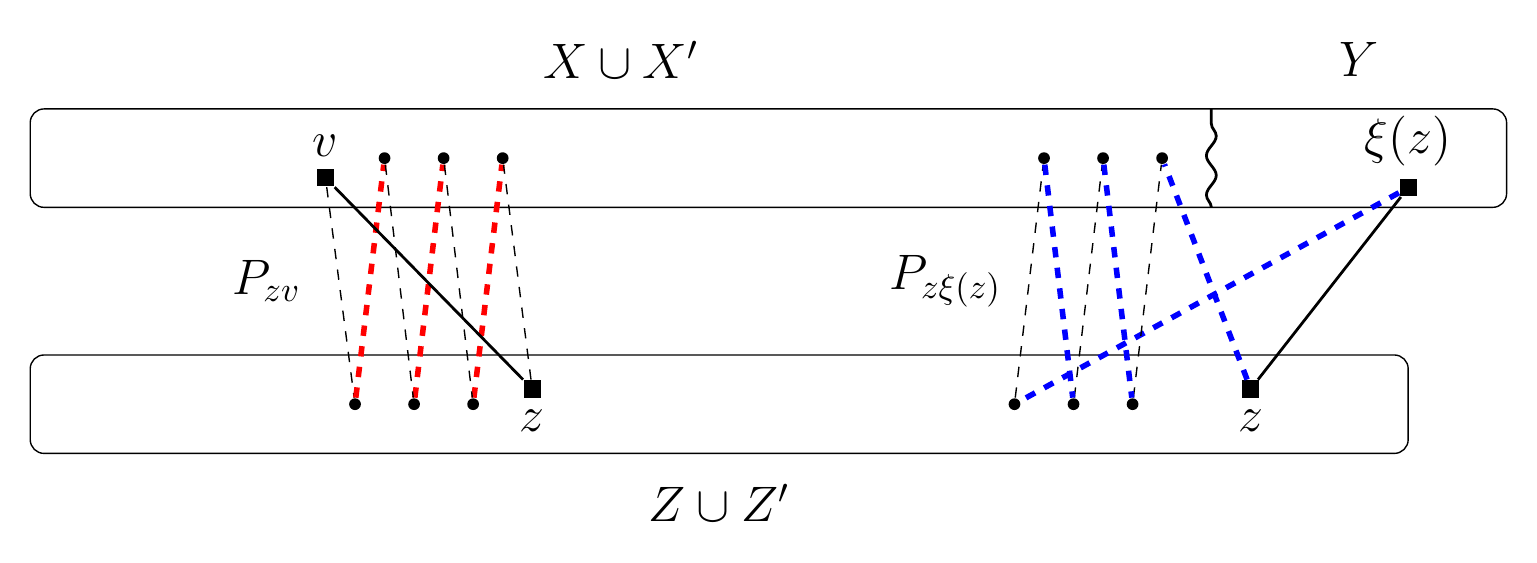}
		\caption{Solid and dashed lines represent edges in $B$ and $B'$, respectively. Note that all internal vertices of every path are outside of $Y$.}
		\label{fig:subdivide}		
	\end{figure}

	As $\xi(Z) = X \cup Y'$ this clearly matches all vertices from $Z \cup X \cup Y'$ and no vertex from $Y \setminus Y'$. Furthermore, every other vertex belongs to some path $P_{zv}$ for $z \in Z$ and $v \in X \cup Y$. As all such vertices are matched as well (regardless of whether we take a blue or a red matching), this
	gives a perfect matching between $Z \cup Z'$ and $X \cup X' \cup Y'$, as required.
\end{proof}

Having the previous two lemmas at hand, we describe our proof strategy. Let $B$ be a graph given by Lemma \ref{lemma:max_deg_3} and consider the $11$-subdivision $B'$ of $B$, as described in Lemma \ref{lemma:subdivision} (11 is, of course, somewhat arbitrary and chosen to make the proof easier). We show that w.h.p $A_{B'} \subseteq \Gnp$ using the following strategy:
\begin{enumerate}
	\item Using Lemma \ref{lemma:factor} embed $|Z|$ copies of a graph $F_\Gamma^+$ into $\Gnp$, where $F_\Gamma^+$ is the graph obtained by adding $3$ new vertices to $F$ and connecting them to all the vertices in $\Gamma$ (see Figure \ref{fig:F_G}). Such copies of $F_\Gamma^+$ correspond to vertices in $Z$ together with their neighbourhood in $B'$ (recall that every vertex in $Z$ has degree $3$ in $B$ and, therefore, in $B'$).

	\item Choose an arbitrary injective mapping of $X \cup Y$ into $V(G)$ which avoids previously found copies of $F_\Gamma^+$.

	\item For each edge $zw \in B$ let $w_z \in X'$ denote the first vertex on the path $P_{zw}$ from $z \in Z$ to $w \in X \cup Y$. Using Lemma \ref{lemma:connecting} we find vertex-disjoint copies of \emph{$F_\Gamma$-paths} (see Figure \ref{fig:F_G_path}) of length $10$ with the endpoints \emph{anchored} in $w$ and $w_z$, one for each edge $zw \in B$. Each such $F_\Gamma$-path corresponds to the remaining vertices on the path $P_{zw}$. This defines an embedding of $A_{B'}$ into $\Gnp$.
\end{enumerate}

\begin{figure}[h!]
	\centering
	\subfloat[The graph $F_\Gamma^+$]{\includegraphics[scale=0.6]{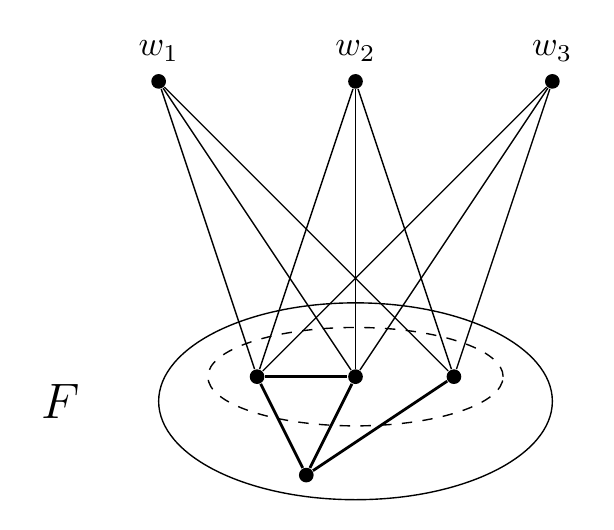} \label{fig:F_G}}
	\hspace{1.5cm}
	\subfloat[The $F_\Gamma$-path of length $6$ from $w$ to $w'$]{\includegraphics[scale=0.6]{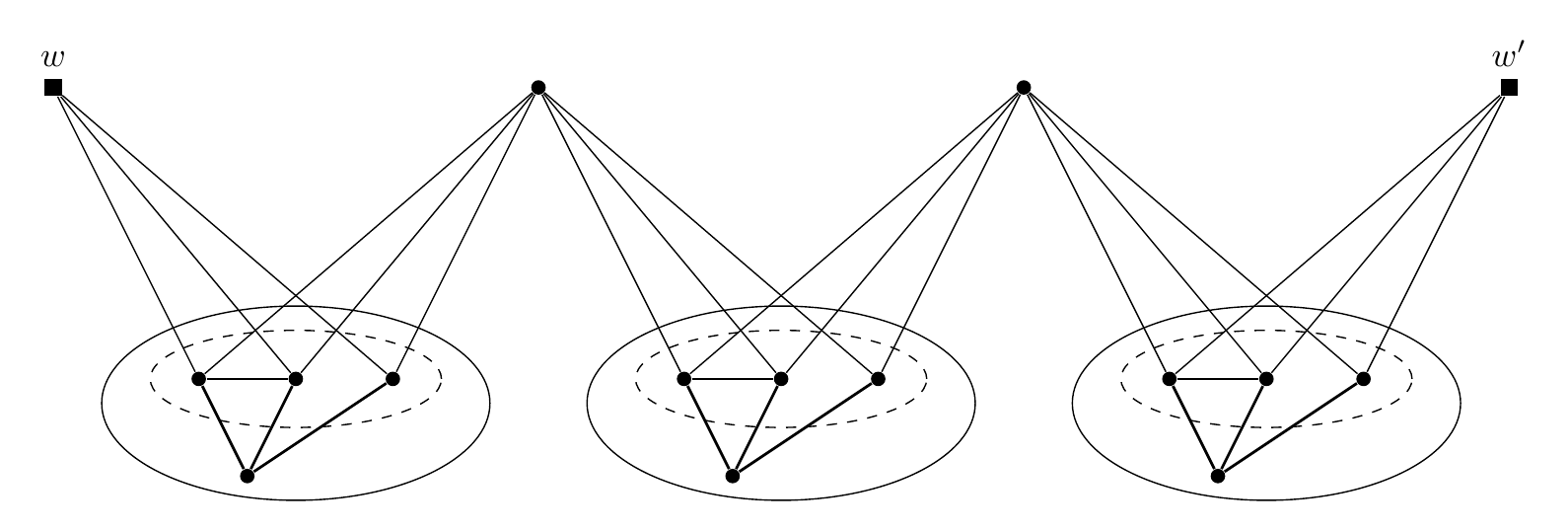} \label{fig:F_G_path}}

	\caption{Building blocks of an absorber $A_B$. Dashed subset of $F$ represents the set $\Gamma$.}
\end{figure}

It is important to notice that in the step 1 we embed $F_\Gamma^+$ rather than just $F$ for the following reason: if we only embed $F$ then, as $|\Gamma| = \Delta$ and $p = n^{-\eps - 1/\Delta}$, the subset corresponding to $\Gamma$ will most likely not have a common neighbourhood. We now make this strategy precise. The following statement together with Lemma \ref{lemma:factor} takes care of the first step. The proof is a rather straightforward case analysis, thus we postpone it for the appendix.

\begin{lemma} \label{lemma:F_plus_m1}
	Let $F$ be a graph with maximum degree $\Delta$ and $\Gamma \subseteq V(F)$ a subset of size $|\Gamma| \le \Delta$, for some integer $\Delta \ge 3$. Then $m_1(F_\Gamma^+) \le \Delta - 1/2$.
\end{lemma}

Next, we define an \emph{$F_\Gamma$-path} more formally. Given a graph $F$ and a subset $\Gamma \subseteq V(F)$, we define the \emph{$F_\Gamma$-path} of length $L$ as follows: Consider a path of length $L$ (that is, a path on $L+1$ vertices), replace each \emph{even} vertex with a distinct copy of $F$ and connect each \emph{odd} vertex (which we call an \emph{outside} vertex) to all vertices in $\Gamma$ from the corresponding neighbouring copies of $F$ (see Figure \ref{fig:F_G_path}). Note that if $L$ is even then $F_\Gamma$-path contains $L/2 + 1$ `outside' vertices.

The following lemma shows that we can apply Lemma \ref{lemma:connecting} with $p = n^{-\eps - 1/\Delta}$ in order to find $F_\Gamma$-paths with anchored endpoints. Again, the proof is a rather simple (but tedious) estimate on the number of edges in various subgraphs of $F_\Gamma$-paths, thus we also postpone it for the appendix.

\begin{lemma} \label{lemma:F_path_m}
	Let $F$ be a graph with maximum degree $\Delta$ and $\Gamma \subseteq V(F)$ a subset of size $|\Gamma| \le \Delta$ such that $\deg_F(v) \le \Delta - 1$ for every $v \in \Gamma$, for some $\Delta \ge 3$. If $H$ is an $F_\Gamma$-path of length $10$ then $m(H, \mathbf{x}) \le \Delta - 1/2$, where $\mathbf{x} = (w, w')$ and $w$ and $w'$ are endpoints of such a path (see Figure \ref{fig:F_G_path}).
\end{lemma}

Finally we are ready to prove Lemma \ref{lemma:delta_absorber}.

\begin{proof}[Proof of Lemma \ref{lemma:delta_absorber}]
	Let $B$ be a bipartite graph on vertex classes $X \cup Y$ and $Z$, as given by Lemma \ref{lemma:max_deg_3} for $m$ to be specified shortly. Consider a graph $B'$ as described in Lemma \ref{lemma:subdivision}. Note that each path (see the construction of $B'$ in Lemma \ref{lemma:subdivision}) adds $5$ vertices to $Z'$ and there is one such path for each edge in $B$, thus
	$$
		|Z \cup Z'| = Cm + 5 e(B) = Cm + 40 |Z| = (C + 120)m.
	$$
	Therefore, we choose $m = \lfloor \nu n / (C + 120) \rfloor$. In order to make $|Z \cup Z'| = \lfloor \nu n \rfloor$ we further `pad' a perfect matching of size $\lfloor \nu n \rfloor - m (C + 120)$ to $B'$. As this is just a minor technicality  assume that $|Z \cup Z'| = \lfloor \nu n \rfloor$. Moreover, both $X \cup X'$ and $Y$ are linear in $m$ which implicitly defines constants $\beta_1, \beta_2$.

	We proceed in two steps. First, from Lemma \ref{lemma:factor} and $m_1(F_\Gamma^+) \le \Delta - 1/2$ (Lemma \ref{lemma:F_plus_m1}) we have that $G_1 = \Gnp$ w.h.p contains a family $\{F_z^+\}_{z \in Z}$ of vertex-disjoint copies of $F_\Gamma^+$. For each $z \in Z$ arbitrarily identify the neighbourhood of $N_{B'}(z)$ with the `outside' vertices of $F_z^+$ and let $F_z \subseteq F_z^+$ be the subgraph which corresponds to $F$. Furthermore, arbitrarily identify $X \cup Y$ with a subset of $V(G)$ (such that all the copies of $F_\Gamma^+$ are avoided). To summarise, this embeds the part of $A_{B'}$ corresponding to vertices in $Z$, $N_{B'}(Z)$ and $X \cup Y$ (refer to the beginning of this section for the description of $A_B$).

	It remains to embed copies of $F$ which correspond to $Z'$ and vertices which correspond to the remaining vertices from $X'$. For each path $P_{zw}$ (for $zw \in B$) from $z$ to $w$ let $w_{z} \in N_B'(z)$ denote the first vertex which comes after $z$. Then the remaining part of $A_{B'}$ corresponds to the vertices in $\bigcup_{zw \in B} V(P_{zw}) \setminus \{z, w, w_{z}\}$. In particular, for each path $P_{zw}$ we have only embedded one edge (the one corresponding $z w_z$), thus it remains to embed the part corresponding to the path of length $10$ from $w_{z}$ to $w$. Such a path corresponds to an $F_\Gamma$-path of length $10$ where the two endpoints are $w_{z}$ and $w$. Finally, from Lemma \ref{lemma:F_path_m} and Lemma \ref{lemma:connecting} we infer that $G_2 = \Gnp$ w.h.p contains the desired family of pairwise-disjoint $F_\Gamma$-paths.

	To summarise, $G = G_1 \cup G_2$ contains a family $\{F_z\}_{z \in Z \cup Z'}$ of copies of $F$ and a mapping $\gamma \colon X \cup X' \cup Y \to V(G)$ which avoid these copies such that $zv \in B'$ implies $\Gamma_z \subseteq N_{G}(\gamma(v))$, where $\Gamma_z \subseteq F_z$ corresponds to a subset $\Gamma \subseteq F$. As $B'$ is $Y$-robust we conclude that the auxiliary bipartite graph $\calB_G(\{\Gamma_z\}_{z \in Z \cup Z'}, \gamma(X \cup X') \cup \gamma(Y))$ is $\gamma(Y)$-robust, thus $G$ satisfies the property of the lemma.
\end{proof}

%% file: concluding.tex
\section{Concluding remarks}

Our main contribution is Theorem \ref{thm:spanning} which shows that a typical $\Gnp$ is $\calH(n, \Delta)$-universal provided $p= \tilde \Omega(n^{-1/(\Delta-1/2)})$. Recall that for such a value of $p$, a typical $\Gnp$ does not have the property that every subset of $\Delta$ vertices has a common neighbourhood --- a feature which is very useful for a `vertex-by-vertex' embedding. We hope that the techniques introduced in this paper will be helpful in breaking such barriers in various embedding-type problems like the Bandwidth Theorem for random graphs \cite{allen2015local}. Moreover, we hope our techniques will be useful in making further progress towards an `optimal' (in a certain sense) sparse blow-up lemma \cite{allen2016blow}.

Our approach exploits the fact that every $\Delta$ regular graph can be made $(\Delta-1)$-degenerate by removing just a few vertices from each component. Therefore, it seems like the `natural' edge-probability we should have obtained is $p = \tilde \Omega(n^{-1/(\Delta - 1)})$ (in which case it would match the best known bound for the almost-spanning case \cite{conlon17almost}). Moreover, such a bound would be optimal for $\Delta = 3$ (see \eqref{eq:JKV bound}). Unfortunately, assumption in Lemma \ref{lemma:delta_S_embedding} does not allow for $p$ smaller than what we have (see Section \ref{sec:weaker_ordering} for a detailed discussion) and our finishing part seems too wasteful. Anyway, it is likely that $n^{-1/(\Delta-1)}$ is the best our method could potentially do and obtaining a universality result for edge probability smaller than $n^{-1/(\Delta - 1)}$, even in the almost-spanning case, remains a formidable challenge.

Finally, for $d$-degenerate graphs, it would be interesting to obtain a bound of order $\tilde \Omega(n^{-1/d})$ for the universality problem. The case $d = 1$ has been recently announced by Montgomery \cite{montgomery2014embedding} and it is wide open even for the case $d = 2$.

%% file: appendix.tex
\section{Proof of Lemma \ref{lemma:expansion_edges}}



We use of the following version of Janson's inequality. This particular
statement follows immediately from Theorems $8.1.1$ and $8.1.2$ in
\cite{alon2004probabilistic}.

\begin{theorem}[Janson's inequality] \label{thm:Janson}
Let $p \in (0, 1)$ and consider a family $\{ H_i \}_{i \in
\mathcal{I}}$ of subgraphs of the complete graph on the vertex set
$[n]$. Let $G = \Gnp$. For each $i \in \mathcal{I}$, let $X_i$
denote the indicator random variable for the event that $H_i \subseteq
G$ and, for each ordered pair $(i, j) \in \mathcal{I} \times
\mathcal{I}$ with $i \neq j$, write $H_i \sim H_j$ if $E(H_i)\cap
E(H_j)\neq \emptyset$. Then, for
\begin{align*}
X &= \sum_{i \in \mathcal{I}} X_i, \\
\mu &= \mathbb{E}[X] = \sum_{i \in \mathcal{I}} p^{e(H_i)}, \\
\delta &= \sum_{\substack{(i, j) \in \mathcal{I} \times \mathcal{I} \\ H_i \sim H_j}} \mathbb{E}[X_i X_j] = \sum_{\substack{(i, j) \in \mathcal{I} \times \mathcal{I} \\ H_i \sim H_j}} p^{e(H_i) + e(H_j) - e(H_i \cap H_j)}
\end{align*}
and any $0 < \gamma < 1$,
$$ \Pr[X < (1 - \gamma)\mu] \le e^{- \frac{\gamma^2 \mu^2}{2(\mu + \delta)}}. $$
\end{theorem}

For convenience of the reader, we restate Lemma \ref{lemma:expansion_edges}.

\begin{customlemma}{\ref{lemma:expansion_edges}}
Let $\Delta \ge 2$ be an integer and $\alpha \in \mathbb{R}$ a positive constant. Given a subset $W \subseteq [n]$ of size $|W| \ge \alpha n$, if
$$
	p \ge \left(n^{-1}\log^3 n \right)^{1/(\Delta - 1/2)}
$$
then $G = \Gnp$ w.h.p has the following property: For every subset $W' \subseteq [n] \setminus W$ and every family $\{(A_i, B_i)\}_{i \in [t]}$ of pairs of subsets $A_i, B_i \subseteq [n] \setminus (W \cup W')$ of size $|A_i| = |B_i| = \Delta - 1$ such that
	\begin{itemize}
		\item $2 t \le |W'|$, and
		\item no vertex of $G$ appears in more than $\Delta$ pairs,
	\end{itemize}
there exists a family of vertex-disjoint edges $\{x_i y_i \in G[W \cup W']\}_{i \in [t]}$ such that $A_i \subseteq N_G(x_i)$ and $B_i \subseteq N_G(y_i)$ for every $i \in [t]$.
\end{customlemma}
\begin{proof}
Consider an arbitrary partition of $W$ into $3\Delta^3 + 1$ subsets denoted by $W_0, W_1, \ldots, W_{3\Delta^3}$ such that $|W_1| = \ldots = |W_{\Delta^3}| = \alpha n / 6 \Delta^3$ and, consequently, $|W_0| \ge \alpha n / 2$. Then $G = \Gnp$ w.h.p satisfies the property of Lemma \ref{lemma:connecting} for every $W_i$ (as $W$) and every graph $F$ consisting of two sets $A, B$ (which may overlap) of size at most $\Delta$  and two adjacent vertices $a, b$ such that $a$ is connected to every vertex in $A$ and $b$ to every vertex in $B$, and $\mathbf{x} = A \cup B$. Moreover, suppose that $G$ has the property described in the following claim:

\begin{claim} \label{claim:matching_fill}
For $p$ as stated in the lemma, $G = \Gnp$ w.h.p has the following property: For every subset $U \subseteq V(G)$ of size $|U| \ge \alpha n / 4$ and every family $\{(A_i, B_i)\}_{i \in [t]}$ of pairs of sets $A_i, B_i \subseteq [n] \setminus U$ of size $|A_i| = |B_i| = \Delta - 1$ such that
\begin{enumerate}[(i)]		
	\item $t \ge n / \log n$, and
	\item $(A_i \cup B_i) \cap (A_{i'} \cup B_{i'}) = \emptyset$ for all $i \neq i'$,
\end{enumerate}
there exists $i \in [t]$ and an edge $xy \in G[U]$ such that $x \in N_G(A_i)$ and $y \in N_G(B_i)$.
\end{claim}

We first show that such $G$ satisfies the property of Lemma \ref{lemma:expansion_edges} and then prove Claim \ref{claim:matching_fill}.

Let $\calP = \{(A_i, B_i)\}_{i \in [t]}$ be a given family of pairs of subsets which satisfy the condition of the lemma. We first partition $[t]$ into subsets $I_1, \ldots, I_{3\Delta^3}$ such that the following holds for every $k \in [3\Delta^3]$ and distinct $i, j \in I_k$:
\begin{enumerate}[(a)]
	\item $|A_i \cap B_i| = |A_j \cap B_j|$, and
	\item $(A_i \cup B_i) \cap (A_j \cup B_j) = \emptyset$.
\end{enumerate}
The existence of such partition can be seen as follows: consider a graph $H$ on the vertex set $[t]$ such that two vertices $i,j$ are adjacent iff $(A_i \cup B_i) \cap (A_j \cup B_j) \neq \emptyset$. As no vertex appears in more than $\Delta$ pairs, we conclude that $H$ has maximum degree at most $2\Delta \cdot (\Delta - 1)$. Therefore, the chromatic number of $H$ is at most $2\Delta^2$ thus there exists a partition of $[t]$ into at most $2\Delta^2$ independent sets. Note if $i, j$ are not adjacent in $H$ then the corresponding pairs satisfy the property (b). Furthermore, partition each independent set into at most $\Delta + 1$ sets depending on the size of of $A_i \cap B_i$. This gives the desired partition of $[t]$.

Let $W_0' := W_0 \cup W'$. For most of the pairs from $\calP$ we find the corresponding edge in $G[W_0']$ and the remaining ones are taken care of using $G[W_i]$. To this end, let $I' \subseteq [t]$ be a maximal subset such that $G[W_0']$ contains a desired family of edges for $\calP' = \{(A_i, B_i)\}_{i \in I'}$. Let us denote with $U \subseteq W_0'$ the subset of `unused' vertices (i.e. vertices which are not part of such edges) and note that $|U| \ge |W_0| \ge \alpha n / 2$. This comes from the fact that $W'$ itself is large enough to accommodate all such edges. We claim that $|I_k \setminus I'| \le n / \log n$ for each $k \in [3\Delta^3]$: if this is not the case then by Claim \ref{claim:matching_fill} there exists some $i \in I_k \setminus I'$ and an edge $xy \in G[U]$ such that $x \in N_G(A_i)$ and $y \in N_G(B_i)$. However, this contradicts the maximality of $I'$. Therefore, for each $k \in [3\Delta^3]$ we have that $I_k' := I_k \setminus I'$ is of size at most $n / \log n$. From the assumption that $G$ satisfies the property of Lemma \ref{lemma:connecting} for $W_k$ (as $W$) we conclude that $G[W_k]$ contains a family of desired edges for $\calP_k := \{(A_i, B_i)\}_{i \in I_k'}$. As $W_k$'s are chosen to be disjoint, this gives a desired family of edges for $\calP$.

It remains to prove Claim \ref{claim:matching_fill}.

\begin{proof}[Proof of Claim \ref{claim:matching_fill}]
We say that a family $\calP = \{(A_i, B_i)\}_{i \in [t]}$ is \emph{valid} if it satisfies (i) and (ii). Given a subset $U$ of size $|U| \ge \alpha n / 4$ and a valid family $\calP = \{(A_i, B_i)\}_{i \in [t]}$ for some $t \ge n / \log n$, let $\calE(\calP, U)$ denote the event ``there is no $i \in [t]$ and an edge $xy \in G[U]$ such that $A_i \subseteq N_G(x)$ and $B_i \subseteq N_G(y)$". The claim then states that no $\calE(\calP, U)$ happens. To prove this it suffices to show $\Pr[\calE(\calP, U)] < e^{-3 \Delta t \log n}$: there are at most $2^n$ subsets $U$ of size $u \ge \alpha n / 4$ and at most $n^{2\Delta t}$ valid families $\calP$ of size $t$, thus by the union bound we get
\begin{align*}
	\Pr[ \exists \text{ valid } \calP, U \colon \calE(\calP, U)] &\leq \sum_{\calP, U} \Pr[\calE(\calP, U)] 
\le \sum_{\calP, U} e^{- 3\Delta |\calP| \log n} \\
	&\le \sum_{\substack{u \ge \alpha n / 4 \\ t \ge n / \log n}} 2^n e^{2\Delta t \log n} e^{-3 \Delta t \log n} = o(1).
\end{align*}

Consider some valid family $\calP = \{(A_i, B_i)\}_{i \in [t]}$ and a subset $U$. We show $\Pr[\calE(\calP, U)] < e^{-3 \Delta t \log n}$ using Janson's inequlaity (Theorem \ref{thm:Janson}) applied on a certain family of subgraphs of the complete graph $K_n$, which we define next. First, let $U = V_1 \cup V_2$ be an arbitrary partition with $|V_1| = |V_2|$. For each $x \in V_1$, $y \in V_2$ and $i \in [t]$, let us define $H_{x,y,i} \in K_n$ to be the subgraph of $K_n$ consisting of the vertex set $U(x,y,i):=\{x, y\} \cup A_i \cup B_i$ and the edge set $E(x,y,i):=\{vw\}\cup\{xa \colon a \in A_i\}\cup\{yb\colon b \in B_i\}$. Let $\calI := V_1 \times V_2 \times [t]$ and recall that $|A_i|=|B_i|=\Delta - 1$. Using the notation from Theorem \ref{thm:Janson} we obtain
$$
	\mu = \sum_{x, y, i} p^{2(\Delta - 1) + 1} = (|U|/2)^2 t p^{2\Delta - 1} \ge t \left( \frac{\alpha n}{8} p^{\Delta - 1/2} \right)^2 \gg t \log n.
$$

Next, we show $\delta = o(\mu^2 / (t \log n))$. Observe that this suffices to apply Janson's inequality with, say, $\gamma = 1/2$, to conclude the probability that none of the $H_{x,y,i}$'s appear in $G$ is at most $e^{-3 \Delta t \log n}$, with room to spare. Recall the definition of $\delta$,
$$
	\delta = \sum_{H_{x,y,i} \sim H_{x',y', i'}} p^{2 \cdot (2(\Delta - 1) + 1) - e(H_{x,y,i} \cap H_{x',y',i'})},
$$
where $H_{x,y,i} \sim H_{x',y',i'}$ if the two subgraphs have a common edge. Because of the property (ii), $x,x' \in V_1$ and $y,y' \in V_2$, for each two such subgraphs we either have (a) $i = i'$ and either $x = x'$ or $y = y'$ (but not both), or (b) $x = x'$, $y = y'$ and $i \neq i'$. Let us denote with $\delta_a$ and $\delta_b$ the contribution of pairs which satisfy (a) and (b), respectively. We first estimate $\delta_a$. Note that for each pair of subgraphs $H_{x,y,i}$ and $H_{x,y',i}$ we have $e(H_{x,y,i} \cap H_{x, y', i}) = \Delta - 1$, thus
$$
	\delta_a \le t (|U|/2)^3 \cdot p^{2(2\Delta - 1) - (\Delta - 1)} \le \frac{2 \mu^2}{t|U| p^{\Delta - 1}} \ll \mu^2 / (t \log n),
$$
with room to spare. Here we used $t (|U|/2)^3$ as an upper bound on the number of choices of 4-tuples $(x,y,i,x')$ and $(x,y,i,y')$. Similarly, from $e(H_{x,y,i} \cap H_{x,y,i'}) = 1$ we obtain
$$
	\delta_b \le (|U|/2)^2 t^2 \cdot p^{2(2\Delta - 1) - 1} =  \frac{4 \mu^2}{|U|^2 p} \ll \mu^2 / (t \log n).
$$
To summarise, we showed $\delta = o(\mu^2 / (t \log n))$ which completes the proof.
\end{proof}
\end{proof}

\section{Proofs of density estimates from Section \ref{sec:delta_absorber}}

For the definition of $F_\Gamma^+$ graphs and $F_\Gamma$-paths we refer the reader to Section \ref{sec:delta_absorber}.

\begin{customlemma}{\ref{lemma:F_plus_m1}}
	Let $F$ be a graph with maximum degree $\Delta$ and $\Gamma \subseteq V(F)$ a subset of size $|\Gamma| \le \Delta$, for some integer $\Delta \ge 3$. Then $m_1(F_\Gamma^+) \le \Delta - 1/2$.
\end{customlemma}
\begin{proof}	
	Given a graph $H$ with at least 2 vertices, let $d_1(H) := e(H)/(v(H)-1)$. Then $m_1(F_\Gamma^+) = \max d_1(H)$, where the maximum is taken over all subgraphs $H \subseteq F_\Gamma^+$ with at least 3 vertices.

	Consider a subgraph $H \subseteq F_\Gamma^+$ and let $a = |V(H) \cap V(F)|$ and $b = |V(H) \cap \{w_1,w_2,w_3\}|$. We first deal with the case $a \in \{1,2,3\}$. If $a = 1$ then $H$ is a star thus $e(H) = v(H) - 1$ and $d_1(H) = 1$. Otherwise, if $a = 2$ then $e(H) \le 1 + 2b$ and
	$$
		d_1(H) \le \frac{1 + 2b}{2 + b - 1} \le 2 < \Delta - 1/2.
	$$
	Finally, if $a = 3$ then $e(H) \le 3 + 3b$ and
	$$
		d_1(H) \le \frac{3 + 3b}{3 + b - 1} < 5/2 \le \Delta - 1/2,
	$$
	for every valid $b$ (that is, for $b \in \{0, 1, 2, 3\}$).

	Let us now assume $a \ge 4$. Then
	$$
		\frac{e(H)}{v(H) - 1} \le \frac{a \Delta / 2 + b\Delta}{a + b - 1}.
	$$
	We aim to show that the right hand side can be upper bounded by $\Delta - 1/2$. Equivalently, we aim to show
	$$
		(a/2 + b)\Delta \le \Delta(a + b - 1) - (a + b - 1)/2,
	$$
	which after rearranging  corresponds to $\Delta(a/2 - 1) \ge (a+b-1)/2$. As $b \le 3$ it suffices to show $\Delta(a/2 - 1) \ge (a + 2)/2$ which holds trivially for every $a\geq 4$. This completes the proof.
\end{proof}

\begin{customlemma}{\ref{lemma:F_path_m}}
	Let $F$ be a graph with maximum degree $\Delta$ and $\Gamma \subseteq V(F)$ a subset of size $|\Gamma| \le \Delta$ such that $\deg_F(v) \le \Delta - 1$ for every $v \in \Gamma$, for some $\Delta \ge 3$. If $H$ is an $F_\Gamma$-path of length $10$ then $m(H, \mathbf{x}) \le \Delta - 1/2$, where $\mathbf{x} = (w, w')$ and $w$ and $w'$ are endpoints of $H$ (see Figure \ref{fig:F_G_path}).
\end{customlemma}
\begin{proof}
	We first argue that $d_1(H') \le \Delta - 1/2$ (where $d_1(\cdot)$ is as defined in the proof of Lemma \ref{lemma:F_plus_m1}) for every subset $H' \subseteq H$ with at least two vertices. This verifies
	$$
		\frac{e(H')}{v(H') - \max\{1, |V(H') \cap \{w, w'\}|\}} \le \Delta - 1/2
	$$
	for all subgraphs $H' \subseteq H$ with $V(H) \cap \{w, w'\} = \emptyset$.

	To this end, note that if $H', H'' \subseteq H$ have only one vertex in common and no edges in between, that is, $V(H') \cap V(H'') = \{h\}$ and $H$ contain no edge between $V(H') \setminus \{h\}$ and $V(H'') \setminus \{h\}$, then $d_1(H' \cup H'') \le \max\{d_1(H'), d_1(H'')\}$. Now $d_1(H') \le \Delta - 1/2$ can be seen as follows: set $H' = \bigcup_{i \in [5]} H' \cap F_i^+$, where $F_i^+$ is the subgraph of $H$ that corresponds to the $i$-th copy of $F$ on the $F_\Gamma$-path together with the two neighbouring vertices (see Figure \ref{fig:subgraph_F_i}). 

	\begin{figure}[h!]
		\centering
		\includegraphics[scale=0.6]{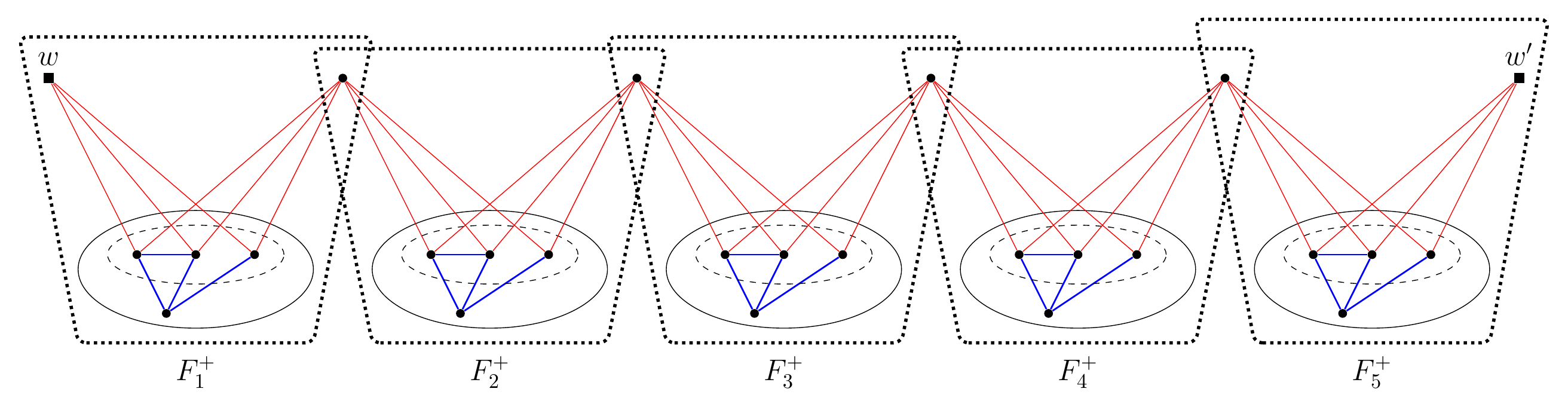}
		\caption{Subgraphs $F_i^+$.}
		\label{fig:subgraph_F_i}
	\end{figure} 

	\noindent
	As $F_i^+ \subseteq F_\Gamma^+$, from Lemma \ref{lemma:F_plus_m1} we have $d_1(F_i^+) \le \Delta - 1/2$. Finally, as every two consecutive $F_i^+$ and $F_{i+1}^+$ intersect on exactly one vertex and otherwise have no edges in between, the previous observation implies $d_1(H') \le \Delta - 1/2$.

	Let us now consider a subgraph $H' \subseteq H$ which contains both $w$ and $w'$. If $H'$ is not connected then $H' = H_1 \cup H_2$ where $H_1$ and $H_2$ are disjoint subgraphs and there are no edges between them. Therefore,
	$$
		\frac{e(H')}{v(H') - 2} = \frac{e(H_1) + e(H_2)}{(v(H_1) - 1) + (v(H_2) - 1)} \le \Delta - 1/2,
	$$
	where the last inequality follows from $d_1(H_1), d_1(H_2) \le \Delta - 1/2$ (shown in the previous case). 

	Finally, it remains to consider the case where $H'$ is connected and contains $\{w,w'\}$. Note that such $H'$ necessarily contains all $6$ `outside' vertices as otherwise it is not connected. We estimate the number of edges of such $H'$ as follows: let $a$ denote the number of vertices of $H'$ which belong to a subset of some copy of $F$ corresponding to $\Gamma$ (i.e. the number of vertices which belong to some dashed subset of $F$ in Figure \ref{fig:F_G_path}) and let $b$ denote all the other vertices of $H'$ which belong to some copy of $F$. Then $H'$ contains at most $\frac{a(\Delta - 1) + b \Delta}{2}$ blue edges (i.e. edges within a copy of some $F$) and at most $2a$ red edges (edges incident to outside vertices). From $v(H') = a + b + 6$ we obtain
	\begin{equation} \label{eq:Hprim}
		e(H') \le \frac{a (\Delta - 1) + b \Delta}{2} + 2a \le 
		\frac{(v(H') - 6)\Delta - a}{2} + 2a = 
		\frac{(v(H') - 6)\Delta + 3a}{2} \le \frac{(v(H') - 6)(\Delta + 3)}{2}. 
	\end{equation}
	On the other hand, we need to show that this is at most
	$$
		(v(H') - 2)(\Delta - 1/2) = \frac{(v(H') - 2)(2\Delta - 1)}{2},
	$$
	which clearly follows from \eqref{eq:Hprim} if $\Delta \ge 4$. Moreover, for $\Delta = 3$ this also holds provided $v(H') \le 26$. Thus it remains to check the case $\Delta = 3$ and $v(H') > 26$. In this case the previously used estimate $2a \le (v(H') - 6)2$ on the number of red edges in $H'$ is too generous as there are only $24$ red edges in total. As $H'$ contains at most $(v(H') - 6)3/2$ blue edges, this gives
	$$
		e(H') \le \frac{(v(H') - 6)3}{2} + 24
	$$
	which is easily seen to be at most $(v(H') - 2) 5 / 2$ for $v(H') > 26$. This concludes the proof.
\end{proof}